\documentclass[10pt]{amsart}

\usepackage{fullpage}

\usepackage{amsmath,amssymb,amsthm}
\usepackage{enumitem}
\usepackage[hidelinks]{hyperref}
\usepackage[nameinlink,capitalize,noabbrev]{cleveref}

\numberwithin{equation}{section}
\theoremstyle{plain}
\newtheorem{theorem}{Theorem}[section]
\newtheorem{lemma}[theorem]{Lemma}
\newtheorem{proposition}[theorem]{Proposition}
\newtheorem{corollary}[theorem]{Corollary}
\newtheorem{claim}[theorem]{Claim}
\theoremstyle{definition}
\newtheorem{definition}[theorem]{Definition}
\theoremstyle{remark}
\newtheorem{remark}[theorem]{Remark}

\newcommand{\R}{\mathbb R}
\newcommand{\N}{\mathbb N}
\newcommand{\Z}{\mathbb Z}
\newcommand{\PP}{\mathbb P}
\newcommand{\EE}{\mathbb E}

\DeclareMathOperator{\rank}{rank}
\DeclareMathOperator{\spann}{span}
\DeclareMathOperator{\supp}{supp}
\DeclareMathOperator{\dist}{dist}
\DeclareMathOperator{\diag}{diag}
\DeclareMathOperator{\conv}{conv}
\DeclareMathOperator{\sgn}{sgn}
\newcommand{\Id}{\mathrm{Id}}

\title[Banach--Mazur distance to the cross-polytope]{An improved lower bound on the Banach--Mazur distance to the cross-polytope}

\author{Omer Friedland}
\address{Institut de Math\'ematiques de Jussieu, Sorbonne Universit\'e, 4 place Jussieu, 75005 Paris, France}
\email{omer.friedland@imj-prg.fr}



\begin{document}

\begin{abstract}
Let $\Gamma$ be an $n\times m$ matrix with independent standard Gaussian entries and let $G_m = \Gamma(B_1^m)$ be the associated Gaussian Gluskin polytope (equivalently, a random $n$-dimensional quotient of $\ell_1^m$). In the regime $m = n^3$ we prove that, with probability at least $1-2/n$,
$$
d_{\mathrm{BM}}(G_m,B_1^n) \ge c n^{4/7}(\log n)^{-C},
$$
where $B_1^n = \conv\{\pm e_1,\dots,\pm e_n\}$ is the cross-polytope. This improves the previously best-known exponent $5/9$ (up to logarithmic factors) for this Gaussian model; in particular, the same lower bound holds for $\sup_{K} d_{\mathrm{BM}}(K,B_1^n)$.

The main new ingredient is a conditioning-compatible treatment of the regime of ``many small-coefficients''. After passing to a suitable Gaussian quotient, we apply a Maurey-type sparsification that reduces the relevant entropy (in effect shrinking the support size from $k$ to $k/\log(n\rho)$) at the cost of a Euclidean thickening. We control this enlargement via a Gaussian measure bound stable under Euclidean thickening. In the complementary regime of ``few small-coefficients'', we give a streamlined argument avoiding the global tilting step in earlier work. Together these ingredients rebalance entropy and small-ball estimates and yield the exponent $4/7$.
\end{abstract}

\maketitle

\tableofcontents

%%%%%%%%%%%%%%%%%%%%%%%%%%%%%%%%%%%%%%%%%%%
%%%%%%%%%%%%%%%%%%%%%%%%%%%%%%%%%%%%%%%%%%%
\section{Introduction}

The Banach--Mazur distance quantifies how far two origin-symmetric convex bodies in $\R^n$ are from being linearly isomorphic up to bounded distortion. We recall the definition.

\begin{definition}[Banach--Mazur distance] \label{def:BM}
Let $K,L\subset\R^n$ be origin-symmetric convex bodies. Their Banach--Mazur distance is
$$
 d_{\mathrm{BM}}(K,L)
 := \inf\Bigl\{\rho\ge 1: \exists T\in GL(n) \text{ such that } T(K)\subset L\subset \rho T(K)\Bigr\}.
$$
\end{definition}

A classical benchmark problem in asymptotic convex geometry is to estimate
$$
\sup_{K} d_{\mathrm{BM}}(K,B_1^n),
$$
where $B_1^n$ is the cross-polytope.

A general upper bound due to Giannopoulos~\cite{Giannopoulos1995} asserts that
$$
\sup_{K} d_{\mathrm{BM}}(K,B_1^n)\le C n^{5/6}.
$$
This improved earlier estimates of Bourgain--Szarek~\cite{BourgainSzarek1988} and Szarek--Talagrand~\cite{SzarekTalagrand1989}; see also Youssef~\cite{Youssef2014} for an improved constant. These questions are closely connected with restricted invertibility~\cite{BourgainTzafriri1987} and proportional Dvoretzky--Rogers factorization~\cite{DvoretzkyRogers1950}; see, e.g.,~\cite{BourgainSzarek1988,SzarekTalagrand1989,Giannopoulos1995,Youssef2014,DavidsonSzarek2001,NaorYoussef2017,FriedlandYoussef2019}.

On the other hand, Szarek~\cite{Szarek1990} proved the lower bound
$$
\sup_{K} d_{\mathrm{BM}}(K,B_1^n)\ge c\sqrt n \log n
$$
for a universal constant $c>0$. His proof introduced a viewpoint that has proved remarkably robust: one exhibits extremal behavior via Gluskin-type random polytopes~\cite{Gluskin1981} together with an $\varepsilon$--net discretization, and complements it with strong probabilistic input for Gaussian matrices.

Fix $n\ge 3$ and let $m := n^3$. Let $g_1,\dots,g_m$ be independent $N(0,\Id_n)$ vectors in $\R^n$ and write $\Gamma = [g_1 \cdots g_m]\in\R^{n\times m}$. We consider the symmetric random polytope
\begin{equation} \label{eq:intro-def-Gm}
G_m := \conv\{\pm g_1,\dots,\pm g_m\} = \Gamma(B_1^m)\subset\R^n,
\end{equation}
which we refer to as the \emph{Gaussian Gluskin polytope}.

Building on Szarek's framework, Tikhomirov~\cite{Tikhomirov2019} showed that in the Gaussian Gluskin model with $m = n^3$ one has
$$
d_{\mathrm{BM}}(G_m,B_1^n)\ge c n^{5/9}(\log n)^{-C}
$$
with positive probability, for universal constants $c,C>0$.

We improve the exponent to $4/7$ (up to logarithmic factors) and show that the bound holds with probability at least $1-2/n$.

\begin{theorem} \label{thm:main}
There exist universal constants $c,C>0$ such that for all sufficiently large integers $n$, letting $m = n^3$ and $G_m$ be defined by \eqref{eq:intro-def-Gm}, one has
$$
\PP\Bigl\{d_{\mathrm{BM}}(G_m, B_1^n)\ge c n^{4/7}(\log n)^{-C}\Bigr\}\ge 1-\frac{2}{n}.
$$
\end{theorem}

\begin{corollary} \label{cor:supK}
There exist universal constants $c,C>0$ such that for all sufficiently large $n$,
$$
\sup_{K} d_{\mathrm{BM}}(K,B_1^n) \ge c n^{4/7}(\log n)^{-C}.
$$
\end{corollary}

\begin{proof}
By Theorem~\ref{thm:main}, the stated lower bound holds for $K = G_m$ with positive probability, hence some realization of $G_m$ satisfies it.
\end{proof}

This provides a further quantitative step toward closing the gap between the upper bound of order $n^{5/6}$ and the available lower bounds.

%%%%%%%%%%%%%%%%%%%%%%%%%%%%%%%%%%%%%%%%%%%
\subsection*{New ingredients}

Our proof follows the discretization and conditioning/powering scheme of Szarek and Tikhomirov. The improved exponent $4/7$ comes from a new treatment of the case where many coefficients lie below the threshold $1/s$: after conditioning and passing to a Gaussian quotient, we reduce the relevant entropy via a Maurey-type sparsification at the cost of a Euclidean thickening, which is then controlled by a thickening-stable Gaussian measure bound. The complementary case (few coefficients below $1/s$) admits a streamlined argument.

\subsubsection*{Many $U$-generators (large-$U$): entropy reduction compatible with conditioning.}

After conditioning on the exposed columns $\Gamma_{S(A)}$, one is led (after passing to an appropriate $k$-dimensional quotient) to large unions of cross-polytopes indexed by the choice of the $k$ ``small-coefficient'' generators (and ultimately by mixed choices of size about $\binom{n}{k}^2$ before quotienting). For a direct union bound, the available Gaussian-measure bounds for a \emph{single} polytope are too weak to offset this entropy (as reflected in the exponent $5/9$ in \cite{Tikhomirov2019}).

Our main new input is an entropy reduction performed \emph{after conditioning} and \emph{inside the quotient}. A Maurey-type sparsification replaces the $k$ generators by a subset of size
$$
t(k)\asymp \frac{k}{\log(n\rho)},
$$
so the union bound runs over $\binom{n}{t(k)}$ choices rather than $\binom{n}{k}$, at the cost of a Euclidean thickening of radius $r_0\asymp L/\sqrt{t(k)}$. The thickening is controlled by a Gaussian measure bound for thickened absolute convex hulls whose dependence enters through $R + \eta\sqrt d$, matching the scale produced by the Maurey step.

\subsubsection*{Few $U$-generators (small-$U$): suppression and block-tail Gram--Schmidt.}

When $k$ is small, each relevant polytope contains at least $n-k$ sparse (large-coefficient) generators. We obtain the required \emph{uniform Gaussian-measure} bounds without the global ``tilting over all permutations'' step of \cite{Tikhomirov2019}. Instead, we combine (i) a random suppression (averaging) lemma and (ii) a block-tail Gram--Schmidt distance estimate uniform over the suppressed block. Together with a deterministic Gram--Schmidt/volume bound, this yields exponentially small Gaussian measure for each relevant polytope.

%%%%%%%%%%%%%%%%%%%%%%%%%%%%%%%%%%%%%%%%%%%
\subsection*{Outline of the proof}

Fix $\rho\ge 1$ and consider the Banach--Mazur event $\mathcal E_\rho$ defined in \eqref{eq:E-rho}.
To prove Theorem~\ref{thm:main} we show that for
$$
\rho \asymp n^{4/7}(\log n)^{-C}
\quad (m = n^3)
$$
one has $\PP(\mathcal E_\rho)\le 2/n$. The argument follows the Szarek--Tikhomirov global architecture (discretization and conditioning/powering), with new input in the large-$U$ regime.

\subsubsection*{1) Discretization and powering: reduce to a uniform Gaussian measure bound.}

On $\mathcal E_\rho$ one may represent the witness cross-polytope through $G_m = \Gamma(B_1^m)$ by a coefficient matrix $A$ with $\ell_1$-bounded, $n$-sparse columns, so that $G_m\subset \rho\Gamma A(B_1^n)$. We discretize $A$ in $\ell_1$ to a finite net $\mathcal A_\varepsilon$ with $\log|\mathcal A_\varepsilon|\lesssim n^2\log(n\rho)$ for $\varepsilon\simeq(\rho n^2)^{-1}$. For each $A\in\mathcal A_\varepsilon$ we expose $\Gamma_{S(A)}$ and view the remaining Gaussian columns as fresh. Tikhomirov's powering identity yields
$$
\PP\bigl(G_m\subset 2\rho \Gamma A(B_1^n)\mid \mathcal F_A\bigr)
\le \gamma_n\bigl(\mathcal K_A(\rho;\Gamma_{S(A)})\bigr)^{N(A)}.
$$
Thus it suffices to construct a high-probability event on which $\gamma_n(\mathcal K_A)\le 1/2$ holds simultaneously for all $A\in\mathcal A_\varepsilon$.

\subsubsection*{2) Bridge and coefficient splitting: reduce to mixed unions.}

Thresholding coefficients at level $1/s$ decomposes each column into a sparse ``$K$'' part and an $\ell_2$-small ``$U$'' part. A bridge lemma reduces $\mathcal K_A$ to a union of mixed cross-polytopes generated by $n$ vectors, of which exactly $k$ come from the $U$-part:
$$
\gamma_n(\mathcal K_A) \lesssim \sum_{k = 0}^n p_k(A;\Gamma_{S(A)}).
$$
We choose a cutoff $\tilde s \ge 1$ and treat $k\le \tilde s$ and $k>\tilde s$ separately.

\subsubsection*{3) Small-$U$ regime ($k\le \tilde s$): suppression and block-tail Gram--Schmidt.}

In this regime each polytope contains many $K$-generators. On a global event we combine random suppression with a block-tail Gram--Schmidt distance estimate and a deterministic Gram--Schmidt/volume bound to obtain $\gamma_n(\text{ one polytope})\le \exp(-cr)$. A union bound over $\binom{n}{k}^2$ choices implies $\sum_{k\le \tilde s}p_k\le 1/4$ provided (up to logarithms)
$$
\rho \lesssim \frac{n}{\sqrt{\tilde s}}.
$$

\subsubsection*{4) Large-$U$ regime ($k>\tilde s$): Maurey sparsification and thickening-stable bound.}

On a global event $\mathcal E^{(U)}$ we have $\|u_i\|_2\le L$ uniformly. After projecting away the span of the chosen $K$-generators we work in dimension $k$. A Maurey-type sparsification reduces $k$ generators to
$$
t(k)\asymp \frac{k}{\log(n\rho)},
$$
at the cost of a Euclidean thickening of radius $r_0\asymp L/\sqrt{t(k)}$. The thickening-stable Gaussian measure bound (Lemma~\ref{lem:thickening}) controls the Gaussian measure of the thickened polytope, and yields $\sum_{k>\tilde s}p_k\le 1/4$ under the polynomial constraint (up to logarithms)
$$
\rho \lesssim \frac{\tilde s^3}{n^2}.
$$

\subsubsection*{5) Optimization.}

Balancing the two constraints gives $\tilde s\asymp n^{6/7}$ and hence $\rho\asymp n^{4/7}$ (up to logarithms), which yields $\PP(\mathcal E_\rho)\le 2/n$ and completes the proof.

\subsection*{Parameter guide (at a glance)}
\begin{center}
\fbox{\begin{minipage}{0.95\linewidth}\small
The proof uses several auxiliary parameters. Here is their role and typical scale
(up to absolute constants and logarithmic slack; precise choices are in
Lemma~\ref{lem:verify-constraints} and \eqref{eq:params-choice-47}).

\begin{tabular}{@{}p{0.22\linewidth}p{0.74\linewidth}@{}}
$\Lambda=\log(n\rho)$
& log/entropy scale in the large-$U$ regime (see \eqref{eq:tdef}).\\
$\tilde s$
& cutoff for the number $k$ of $U$-generators: $k\le \tilde s$ (small-$U$) vs. $k>\tilde s$ (large-$U$);
optimized $\tilde s\asymp n^{6/7}\Lambda^{2/7}$ (see \eqref{eq:balanced-tilde-s-47}).\\
$r$
& suppression block size in the small-$U$ regime; chosen $r\asymp \tilde s \Lambda$ (see \eqref{eq:params-choice-47}).\\
$s$
& threshold parameter for the $K/U$ coefficient split at level $1/s$; chosen $s\asymp \tilde s^2/(n\Lambda)$
(see \eqref{eq:params-choice-47}).\\
$L$
& uniform bound on $\|u_i\|_2$ on the event $\mathcal E^{(U)}$:
$L=C_0\sqrt{\frac{n}{s}\Lambda}\asymp \frac{n\Lambda}{\tilde s}$ (see \eqref{eq:L-def}).\\
$t(k)$
& Maurey sparsification size in the $k$-dimensional quotient:
$t(k)=\lceil k/\Lambda\rceil$ (see \eqref{eq:tdef}).\\
$r_0$
& Euclidean thickening radius produced by Maurey:
$r_0=L/\sqrt{t(k)}$ (see \eqref{eq:tdef}).\\
$\varepsilon$
& discretization mesh for the coefficient net: $\varepsilon=(\rho n^2)^{-1}$ (see Lemma~\ref{lem:discretization} and below).\\
\end{tabular}
\end{minipage}}
\end{center}

%%%%%%%%%%%%%%%%%%%%%%%%%%%%%%%%%%%%%%%%%%%
\subsection*{Organization of the paper}

Section~\ref{sec:disc} carries out the discretization and reduction to a finite net of sparse coefficient matrices. Section~\ref{sec:powering} recalls the conditioning/powering framework of \cite{Tikhomirov2019}. Section~\ref{sec:bridge} introduces the $K/U$ decomposition and the reduction to unions of mixed cross-polytopes. Section~\ref{sec:small} treats the small-$U$ regime via random suppression and block-tail Gram--Schmidt bounds, with the uniform suppression event proved in Appendix~\ref{app:suppr-proof}. Section~\ref{sec:large} develops the new large-$U$ argument via Maurey sparsification and the thickening-stable estimate. Finally, Section~\ref{sec:completion} optimizes parameters and completes the proof.

Throughout, $c,C,c',C',\dots$ denote positive absolute constants whose values may change from line to line.

%%%%%%%%%%%%%%%%%%%%%%%%%%%%%%%%%%%%%%%%%%%
%%%%%%%%%%%%%%%%%%%%%%%%%%%%%%%%%%%%%%%%%%%
\section{Preliminaries}

%%%%%%%%%%%%%%%%%%%%%%%%%%%%%%%%%%%%%%%%%%%
\subsection{Discretization: a net of sparse coefficient matrices} \label{sec:disc}

We defer a number of standard auxiliary results (elementary $\ell_1$ geometry,
Gaussian concentration/tails, and an $\ell_1$-sparsity linear-programming principle)
to Appendix~\ref{app:standard}. We only record notation used later.

For $k\in\N$ and $x\in\R^k$ we write $\supp(x) := \{i\in[k]: x_i\neq 0\}$.

We work with the Banach--Mazur event
\begin{equation} \label{eq:E-rho}
\mathcal E_\rho := \Bigl\{d_{\mathrm{BM}}(G_m, B_1^n)\le \rho\Bigr\}.
\end{equation}

For $m,n\in\N$ and $\varepsilon>0$, define $\mathcal A_{m,n}$ to be the set of matrices $A\in\R^{m\times n}$ such that for every column $A^{(i)}$,
$$
\|A^{(i)}\|_1\le 1
\quad\text{ and}\quad
|\supp(A^{(i)})|\le n.
$$
Define the discretized subclass
$$
\mathcal A_{m,n}(\varepsilon) := \{A\in\mathcal A_{m,n}: A_{jk}\in \varepsilon\Z, \forall j,k\}.
$$

%%%%%%%%%%%%%%%%%%%%%%%%%%%%%%%%%%%%%%%%%%%
\subsection{From Banach--Mazur to a sparse coefficient matrix}

\begin{lemma}[Sparse coefficient representation on $\mathcal E_\rho$] \label{lem:BM-to-A}
Assume $\mathcal E_\rho$ holds. Then there exists $A\in\mathcal A_{m,n}$ such that
\begin{equation} \label{eq:Erho-coeff}
G_m\subset \rho \Gamma A(B_1^n).
\end{equation}
\end{lemma}

\begin{proof}
By definition of $\mathcal E_\rho$, there exists $T\in GL(n)$ such that
$$
T(G_m)\subset B_1^n \subset \rho T(G_m).
$$
Applying $T^{-1}$ gives
$$
G_m \subset T^{-1}(B_1^n) \subset \rho G_m.
$$
Set
$$
P := \rho^{-1}T^{-1}(B_1^n).
$$
Then $P$ is an origin-symmetric cross-polytope and satisfies $P\subset G_m \subset \rho P$.

Write $P = \conv\{\pm p_1,\dots,\pm p_n\}$ where $p_1,\dots,p_n\in\R^n$ are its vertices. Since $P\subset G_m = \Gamma(B_1^m)$,
each $p_i$ lies in $\Gamma(B_1^m)$.
Apply Lemma~\ref{lem:l1-sparse-existence} with $M = \Gamma$ and $y = p_i$:
there exists $a_i\in\R^m$ such that
$$
\Gamma a_i = p_i,\quad \|a_i\|_1\le 1,\quad |\supp(a_i)|\le \rank(\Gamma)\le n.
$$
Let $A\in\R^{m\times n}$ be the matrix with columns $a_1,\dots,a_n$. Then $A\in\mathcal A_{m,n}$.

Moreover,
$$
\Gamma A(B_1^n)
 = \Gamma\Bigl(\conv\{\pm a_1,\dots,\pm a_n\}\Bigr)
 = \conv\{\pm \Gamma a_1,\dots,\pm \Gamma a_n\}
 = \conv\{\pm p_1,\dots,\pm p_n\} = P.
$$
Thus $G_m\subset \rho P = \rho \Gamma A(B_1^n)$, proving \eqref{eq:Erho-coeff}.
\end{proof}

%%%%%%%%%%%%%%%%%%%%%%%%%%%%%%%%%%%%%%%%%%%
\subsection{Discretization of the coefficient matrix}

\begin{lemma}[Discretization net and reduction] \label{lem:discretization}
Assume $n\ge 2$, $n \le m\le n^{10}$, and $0<\varepsilon\le 1$ satisfies
\begin{equation} \label{eq:disc-epscond}
\varepsilon \rho n^2\le 1.
\end{equation}
Define the finite set $\mathcal A_\varepsilon := \mathcal A_{m,n}(\varepsilon)$. Then the following hold:
\begin{enumerate}[label = \textnormal{(\roman*)}, leftmargin = 3.2em]
\item Every $A\in\mathcal A_\varepsilon$ satisfies $\|A^{(i)}\|_1\le 1$ and
$|\supp(A^{(i)})|\le n$ for all $i\in[n]$, and has entries in $\varepsilon\Z$.

\item One has the bound
\begin{equation} \label{eq:cardAeps}
|\mathcal A_\varepsilon| \le \Big(\frac{C m}{n \varepsilon}\Big)^{n^2}
\end{equation}
for a universal constant $C>0$.
In particular, when $m = n^3$ and $\varepsilon = 1/(\rho n^2)$,
\begin{equation} \label{eq:cardAeps-log}
\log|\mathcal A_\varepsilon| \le C n^2\log(n\rho).
\end{equation}

\item One has the event inclusion
$$
\mathcal E_\rho
\subset
\Bigl\{\exists A\in\mathcal A_\varepsilon:\quad G_m\subset 2\rho \Gamma A(B_1^n)\Bigr\}.
$$
In particular,
\begin{equation} \label{eq:disc-prob}
\PP(\mathcal E_\rho)
\le
\PP\Bigl(\exists A\in\mathcal A_\varepsilon:\quad
G_m\subset 2\rho \Gamma A(B_1^n)\Bigr).
\end{equation}
\end{enumerate}
\end{lemma}

\begin{proof}
Define $\mathcal A_\varepsilon := \mathcal A_{m,n}(\varepsilon)$.
This set is finite because each entry lies in $\varepsilon\Z\cap[-1,1]$ and each column has support size $\le n$.
Item (i) is immediate from the definition.

\subsubsection*{Step 1: reduction to a (non-discretized) coefficient matrix.}

Lemma~\ref{lem:BM-to-A} shows that $\mathcal E_\rho$ implies the event
$$
\mathcal E_\rho^{\mathrm{coeff}}
 := 
\Bigl\{\exists A\in\mathcal A_{m,n}: G_m\subset \rho \Gamma A(B_1^n)\Bigr\},
$$
hence
\begin{equation} \label{eq:Erho-to-Ecoeff}
\PP(\mathcal E_\rho)\le \PP(\mathcal E_\rho^{\mathrm{coeff}}).
\end{equation}

\subsubsection*{Step 2: discretization and absorption.}

Define $q_\varepsilon:\R\to \varepsilon\Z$ by
$$
q_\varepsilon(t) := \varepsilon \sgn(t)\Big\lfloor \frac{|t|}{\varepsilon}\Big\rfloor,
\quad \text{ with }\sgn(0) := 0,
$$
and for $a\in\R^m$ define $Q_\varepsilon(a)\in(\varepsilon\Z)^m$ by $(Q_\varepsilon(a))_j := q_\varepsilon(a_j)$.
Then $|q_\varepsilon(a_j)|\le |a_j|$ and $|a_j-q_\varepsilon(a_j)|<\varepsilon$.

If $\|a\|_1\le 1$ and $|\supp(a)|\le n$, then
$$
\|Q_\varepsilon(a)\|_1\le 1,\quad |\supp(Q_\varepsilon(a))|\le n,\quad \|a-Q_\varepsilon(a)\|_1\le n\varepsilon.
$$
For $A\in\mathcal A_{m,n}$ define $\widetilde A := Q_\varepsilon(A)$ columnwise. Then $\widetilde A\in\mathcal A_{m,n}(\varepsilon)$ and for all $x\in B_1^n $,
$$
\|(A-\widetilde A)x\|_1
\le \sum_{i = 1}^n |x_i| \|A^{(i)}-\widetilde A^{(i)}\|_1
\le (n\varepsilon)\sum_{i = 1}^n |x_i|
\le n\varepsilon.
$$
Hence
$$
A(B_1^n)\subset \widetilde A(B_1^n) + n\varepsilon B_1^m. 
$$
Applying $\Gamma$ and using $G_m = \Gamma(B_1^m)$ yields
$$
\Gamma A(B_1^n)\subset \Gamma\widetilde A(B_1^n) + n\varepsilon G_m.
$$

Assume $\mathcal E_\rho^{\mathrm{coeff}}$ holds and fix $A$ such that $G_m\subset \rho \Gamma A(B_1^n)$. Let $\widetilde A := Q_\varepsilon(A)\in\mathcal A_\varepsilon$. Then
$$
G_m\subset \rho \Gamma\widetilde A(B_1^n) + (\rho n\varepsilon) G_m.
$$
Set $\delta := \rho n\varepsilon$. By \eqref{eq:disc-epscond} and $n\ge 2$,
$$
\delta = \rho n\varepsilon \le \frac{\rho n^2\varepsilon}{n}\le \frac{1}{n}\le \frac12.
$$

Since $G_m$ and $\rho\Gamma\widetilde A(B_1^n)$ are origin-symmetric convex bodies and $\delta<1$, and since
$K\subset L + \delta K$ with $0\le\delta<1$ implies $K\subset (1-\delta)^{-1}L$ (apply the Minkowski functional of $K$),
applying this with $K := G_m$ and $L := \rho\Gamma\widetilde A(B_1^n)$ yields
$$
G_m\subset \frac{1}{1-\delta} \rho\Gamma\widetilde A(B_1^n).
$$
Because $\delta\le 1/2$, one has $(1-\delta)^{-1}\le 2$, hence
$$
G_m\subset 2\rho \Gamma\widetilde A(B_1^n).
$$
Thus $\mathcal E_\rho^{\mathrm{coeff}}$ implies $\{\exists A\in\mathcal A_\varepsilon: G_m\subset 2\rho \Gamma A(B_1^n)\}$, and combining with \eqref{eq:Erho-to-Ecoeff} yields item (iii).

\subsubsection*{Step 3: cardinality bound.}

Fix one column. Since $m\ge n$, the number of possible supports $J\subset[m]$ with $|J|\le n$ equals
$$
\sum_{r = 0}^n \binom{m}{r}.
$$

We claim that there exists a universal constant $C>0$ such that
\begin{equation} \label{eq:support-count-disc}
\sum_{r = 0}^n \binom{m}{r}\le \Big(\frac{C m}{n}\Big)^n.
\end{equation}
Indeed, split into two cases.

If $m\ge 2n$, then $n\le m/2$ and the binomial coefficients are increasing for $0\le r\le n$, hence
$$
\sum_{r=0}^n \binom{m}{r}\le (n+1)\binom{m}{n}.
$$
By Lemma~\ref{lem:binombound}, $\binom{m}{n}\le (em/n)^n$, and since $n+1\le 2^n$ for $n\ge 1$ we obtain
$$
\sum_{r=0}^n \binom{m}{r}\le (n+1)\Big(\frac{em}{n}\Big)^n \le \Big(\frac{2e m}{n}\Big)^n.
$$

If $n\le m<2n$, then trivially $\sum_{r=0}^n \binom{m}{r}\le 2^m \le 2^{2n}=4^n$, and since $m/n\ge 1$ this gives
$$
\sum_{r=0}^n \binom{m}{r}\le 4^n \le \Big(\frac{4m}{n}\Big)^n.
$$
Combining the two cases yields \eqref{eq:support-count-disc} with a universal constant $C>0$ (e.g. $C=6$).

For a fixed support $J$ with $|J| = r\le n$, each supported coordinate lies in $\varepsilon\Z\cap[-1,1]$, hence has at most $C/\varepsilon$ choices.
Thus for that fixed $J$ there are at most $(C/\varepsilon)^r\le (C/\varepsilon)^n$ possibilities. Therefore the number of possible columns is at most
$(C m/(n\varepsilon))^n$. Since there are $n$ columns,
$$
|\mathcal A_\varepsilon|
\le \Big(\frac{C m}{n\varepsilon}\Big)^{n^2},
$$
which is \eqref{eq:cardAeps}. The logarithmic bound \eqref{eq:cardAeps-log} follows when $m = n^3$ and $\varepsilon^{-1} = \rho n^2$.
\end{proof}

\begin{remark}
The discretization step is carried out directly in $\ell_1$: the rounding map is monotone in absolute value, hence preserves the $\ell_1$ constraint, and the resulting perturbation is absorbed using only the inclusion $G_m = \Gamma(B_1^m)$. In particular, the reduction in Lemma~\ref{lem:discretization}(iii) is a deterministic event inclusion and does not require a separate geometric event (such as an inradius lower bound for $G_m$) to absorb the discretization error; compare with the discretization step in \cite[§2]{Tikhomirov2019}.
\end{remark}

%%%%%%%%%%%%%%%%%%%%%%%%%%%%%%%%%%%%%%%%%%%
\subsection{Exposed indices and the powering reduction} \label{sec:powering}

The conditioning/powering reduction is essentially Tikhomirov's Lemma 3.1 framework \cite[§3]{Tikhomirov2019}.

Let $\Gamma = [g_1 \cdots g_m]\in\R^{n\times m}$ as above. For $J\subset[m]$, write $\Gamma_J\in\R^{n\times |J|}$ for the submatrix consisting of columns $(g_j)_{j\in J}$, and for a deterministic $A\in\R^{m\times n}$ write $A_J\in\R^{|J|\times n}$ for the restriction of $A$ to the rows in $J$.

For $A\in\mathcal A_{m,n}$ define the \textbf{exposed index set}
$$
S(A) := \bigcup_{i = 1}^n \supp(A^{(i)})\subset[m].
$$
Since each column has support at most $n$,
\begin{equation} \label{eq:SA-size}
|S(A)|\le n^2.
\end{equation}

Note that, by definition of $S(A)$, all rows of $A$ outside $S(A)$ are zero. Hence for every realization of $\Gamma$,
\begin{equation} \label{eq:GammaA-restrict}
\Gamma A = \Gamma_{S(A)}A_{S(A)}.
\end{equation}

For $A\in\mathcal A_\varepsilon$ define the exposed $\sigma$-field
$$
\mathcal F_A := \sigma(g_j: j\in S(A)).
$$
Conditional on $\mathcal F_A$, the columns $\{g_j: j\notin S(A)\}$ are i.i.d. $N(0,\Id_n)$ and independent of $\mathcal F_A$.

Fix $\rho\ge 1$. For $A\in\mathcal A_{m,n}$ and a realization $\omega$ of the exposed matrix $\Gamma_{S(A)}$, define the (exposed, deterministic) target set
\begin{equation} \label{eq:KAdef}
\mathcal K_A(\rho;\omega) := 2\rho \Gamma_{S(A)}(\omega) A_{S(A)}(B_1^n)\subset\R^n.
\end{equation}

\begin{lemma}[Powering] \label{lem:powering}
Fix $\rho\ge 1$ and $A\in\mathcal A_{m,n}$. Let $\omega = \Gamma_{S(A)}$ be the exposed part.

If
\begin{equation} \label{eq:discAevent}
G_m = \Gamma(B_1^m)\subset 2\rho \Gamma A(B_1^n),
\end{equation}
then every column $g_j$ satisfies $g_j\in \mathcal K_A(\rho;\omega)$.

Moreover, writing $N(A) := m-|S(A)|$, one has the exact identity
\begin{equation} \label{eq:powering-identity}
\PP\bigl(\eqref{eq:discAevent} \big| \mathcal F_A\bigr)
 = 
{\mathbf 1}_{\{\forall j\in S(A): g_j\in \mathcal K_A(\rho;\omega)\}} 
\gamma_n\bigl(\mathcal K_A(\rho;\omega)\bigr)^{N(A)},
\end{equation}
where $\gamma_n$ denotes the standard Gaussian measure on $\R^n$.
In particular,
\begin{equation} \label{eq:powering-cond}
\PP\bigl(\eqref{eq:discAevent} \big| \mathcal F_A\bigr)\le \gamma_n\bigl(\mathcal K_A(\rho;\omega)\bigr)^{N(A)}.
\end{equation}
\end{lemma}

\begin{proof}
By symmetry and convexity, \eqref{eq:discAevent} holds if and only if $g_j\in 2\rho \Gamma A(B_1^n)$ for all $j\in[m]$.
Since $S(A)$ contains the support of every column of $A$, using \eqref{eq:GammaA-restrict}, we have $\Gamma A = \Gamma_{S(A)}A_{S(A)}$, so the right-hand side equals $\mathcal K_A(\rho;\omega)$, proving the first claim.

Condition on $\mathcal F_A$. Then $\mathcal K_A(\rho;\omega)$ is deterministic, the exposed columns $(g_j)_{j\in S(A)}$ are fixed,
and the fresh columns $(g_j)_{j\notin S(A)}$ are i.i.d. $N(0,\Id_n)$ and independent of $\mathcal F_A$.
Therefore, \eqref{eq:discAevent} occurs iff (i) all exposed columns lie in $\mathcal K_A(\rho;\omega)$ and (ii) all $N(A)$ fresh columns lie in it.
This yields \eqref{eq:powering-identity}. The inequality \eqref{eq:powering-cond} follows immediately.
\end{proof}

%%%%%%%%%%%%%%%%%%%%%%%%%%%%%%%%%%%%%%%%%%%
\subsection{$K/U$ decomposition and a union of cross-polytopes} \label{sec:bridge}

We use the same sparse/short (``two-type'') decomposition of coefficient vectors as in \cite[§3]{Tikhomirov2019}.

Fix an integer parameter $s\in[1,n]$ and define a threshold $1/s$. For $a\in\R^m$ define its $K$- and $U$-parts by
$$
a^{(K)}_j := a_j \mathbf 1_{\{|a_j|\ge 1/s\}},
\quad
a^{(U)}_j := a_j \mathbf 1_{\{|a_j|< 1/s\}},
\quad
a = a^{(K)} + a^{(U)}.
$$
For a matrix $A\in\R^{m\times n}$ define $A^{(K)},A^{(U)}$ columnwise and set
$$
F_1(A) := A^{(K)},\quad F_2(A) := A^{(U)},\quad
F(A) := [F_1(A) F_2(A)]\in\R^{m\times 2n}.
$$

If $\|a\|_1\le 1$ then $|\supp(a^{(K)})|\le s$ because each nonzero entry of $a^{(K)}$ has magnitude $\ge 1/s$.
Also, if $\|a\|_1\le 1$ then
\begin{equation} \label{eq:U-l2}
\|a^{(U)}\|_2^2\le \|a^{(U)}\|_\infty\cdot \|a^{(U)}\|_1\le \frac{1}{s}\cdot 1 = \frac{1}{s}.
\end{equation}

Since $S(A)$ contains the support of every column of $A$, all rows of $A$ outside $S(A)$ are zero; the same holds for $F_1(A),F_2(A),F(A)$.
Hence, for every $A\in\mathcal A_\varepsilon$ and every realization of $\Gamma$,
\begin{equation} \label{eq:GammaF-restrict}
\Gamma F(A) = \Gamma_{S(A)}F_{S(A)}(A),
\end{equation}
where $F_{S(A)}(A)\in\R^{|S(A)|\times 2n}$ denotes the restriction of $F(A)$ to rows in $S(A)$.

Fix $A\in\mathcal A_\varepsilon$ and an exposed realization $\omega = \Gamma_{S(A)}$.
Define vectors
$$
v_i = v_i(A;\omega) := \Gamma_{S(A)}(\omega) F_{1,S(A)}(A)e_i\in\R^n,\quad i\in[n],
$$
$$
u_i = u_i(A;\omega) := \Gamma_{S(A)}(\omega) F_{2,S(A)}(A)e_i\in\R^n,\quad i\in[n].
$$
Enumerate the $2n$ vectors by
$$
w_i := v_i\quad(1\le i\le n),
\quad
w_{n + i} := u_i\quad(1\le i\le n).
$$

For $I\subset[2n]$ with $|I| = n$, define the corresponding symmetric cross-polytope
\begin{equation} \label{eq:PIdef}
\mathcal P_I(A;\omega) := \conv\bigl(\{\pm w_i(A;\omega): i\in I\}\bigr).
\end{equation}
Define the full union and its $k$-subunions:
\begin{equation} \label{eq:P-unions}
\mathcal P(A;\omega) := \bigcup_{\substack{I\subset[2n], |I| = n}}\mathcal P_I(A;\omega),
\quad
\mathcal P^{(k)}(A;\omega) := 
\bigcup_{\substack{I\subset[2n], |I| = n\\ |I\cap(n + [n])| = k}}\mathcal P_I(A;\omega),
\end{equation}
so that $\mathcal P(A;\omega) = \bigcup_{k = 0}^n \mathcal P^{(k)}(A;\omega)$.

\begin{lemma}[Bridge lemma] \label{lem:bridge}
For every $A\in\mathcal A_\varepsilon$ and every exposed realization $\omega$,
\begin{equation} \label{eq:bridge-inclusion}
\mathcal K_A(\rho;\omega)\subset 4\rho \mathcal P(A;\omega).
\end{equation}
Consequently,
\begin{equation} \label{eq:innerprob-def}
\gamma_n\bigl(\mathcal K_A(\rho;\omega)\bigr)\le p(A;\omega),
\quad
p(A;\omega) := \gamma_n\bigl(4\rho \mathcal P(A;\omega)\bigr),
\end{equation}
and for $p_k(A;\omega) := \gamma_n\bigl(4\rho \mathcal P^{(k)}(A;\omega)\bigr)$,
\begin{equation} \label{eq:p-split}
p(A;\omega)\le \sum_{k = 0}^n p_k(A;\omega).
\end{equation}
\end{lemma}

\begin{proof}
Fix $A$ and $\omega$. By definition,
$$
\mathcal K_A(\rho;\omega) = 2\rho \Gamma_{S(A)}(\omega) A_{S(A)}(B_1^n).
$$
Let $x\in A_{S(A)}(B_1^n)$. Then $x = A_{S(A)}\xi$ for some $\xi\in B_1^n $ and
$$
x = A_{S(A)}^{(K)}\xi + A_{S(A)}^{(U)}\xi
 = F_{1,S(A)}(A)\xi + F_{2,S(A)}(A)\xi
\in F_{S(A)}(A)\bigl(B_1^n \times B_1^n \bigr).
$$
Since $\|(\xi,\eta)\|_1 = \|\xi\|_1 + \|\eta\|_1$, one has $B_1^n\times B_1^n\subset 2B_1^{2n}$, hence
$$
A_{S(A)}(B_1^n)\subset 2 F_{S(A)}(A)({B_1^{2n}}).
$$
Applying $\Gamma_{S(A)}(\omega)$ and multiplying by $2\rho$ yields
$$
\mathcal K_A(\rho;\omega)\subset 4\rho \Gamma_{S(A)}(\omega) F_{S(A)}(A)({B_1^{2n}}).
$$
Set
$$
W := \Gamma_{S(A)}(\omega) F_{S(A)}(A)\in\R^{n\times 2n}.
$$
Then the right-hand side equals $4\rho W({B_1^{2n}})$, and the columns of $W$ are precisely $\{w_i\}_{i = 1}^{2n}$.

Let $d := \rank(W)\le n$. Applying Corollary~\ref{cor:l1-decomp} to $M = W$ (with $N = 2n$) gives
$$
W({B_1^{2n}})
\subset
\bigcup_{\substack{J\subset[2n], |J|\le d, \rank(W_J) = |J|}}
\conv\{\pm We_j: j\in J\}.
$$
Fix such a set $J$. Choose any $I\subset[2n]$ with $J\subset I$ and $|I| = n$ (possible since $|J|\le d\le n$).
By monotonicity of absolute convex hulls,
$$
\conv\{\pm We_j: j\in J\}\subset \conv\{\pm We_i: i\in I\} = \mathcal P_I(A;\omega).
$$
Therefore $W({B_1^{2n}})\subset \mathcal P(A;\omega)$, and hence
$$
\mathcal K_A(\rho;\omega)\subset 4\rho W({B_1^{2n}})\subset 4\rho \mathcal P(A;\omega),
$$
which is \eqref{eq:bridge-inclusion}. Taking Gaussian measure and using \eqref{eq:p-split} yields the remaining claims.
\end{proof}

\begin{remark}
Instead of applying Carath\'eodory pointwise, we use a rank-aware $\ell_1$-decomposition for $W(B_1^{2n})$ to obtain the set inclusion \eqref{eq:bridge-inclusion} in one step; compare with \cite[§3]{Tikhomirov2019}.
\end{remark}

%%%%%%%%%%%%%%%%%%%%%%%%%%%%%%%%%%%%%%%%%%%
\subsection{Global norm control event for $U$-vectors} \label{sec:global-U}

For the rest of the paper we fix $m = n^3$ and $\varepsilon^{-1} = \rho n^2$ (as in Lemma~\ref{lem:discretization}).

The $U$-part coefficient vectors satisfy the Euclidean bound \eqref{eq:U-l2}. We discretize the class of possible $U$-coefficients and control $\|\Gamma u\|_2$ uniformly.

\begin{definition}[Discrete coefficient class for $U$] \label{def:Ueps}
Let $\mathcal U_\varepsilon$ be the set of all vectors $u\in(\varepsilon\Z)^m$ such that
$$
|\supp(u)|\le n,
\quad
\|u\|_2^2\le \frac{1}{s}.
$$
\end{definition}

\begin{lemma}[Cardinality of $\mathcal U_\varepsilon$] \label{lem:Ueps-card}
Assume $m = n^3$ and $\varepsilon^{-1} = \rho n^2$ with $\rho\ge 1$. Then
$$
\log|\mathcal U_\varepsilon|\le C n\log(n\rho)
$$
for a universal constant $C>0$.
\end{lemma}

\begin{proof}
Since $s\ge 1$, any $u\in\mathcal U_\varepsilon$ satisfies $\|u\|_\infty\le \|u\|_2\le 1/\sqrt{s}\le 1$, hence every coordinate of $u$ lies in $\varepsilon\Z\cap[-1,1]$. We overcount $|\mathcal U_\varepsilon|$ by ignoring the remaining $\ell_2$ restriction beyond this crude coordinate bound. For each $1\le r\le n$, the number of supports of size $r$ is $\binom{m}{r}\le (em/r)^r\le (em)^r$. Given a fixed support of size $r$, each nonzero coordinate must lie in $\varepsilon\Z\cap[-1,1]$, hence has at most $C/\varepsilon$ choices. Thus the number of vectors with support size $r$ is at most $(C m/\varepsilon)^r$. Summing over $0\le r\le n$ yields
$$
|\mathcal U_\varepsilon|\le \sum_{r = 0}^n (C m/\varepsilon)^r \le (n + 1)(C m/\varepsilon)^n \le (C m/\varepsilon)^n,
$$
absorbing the factor $n + 1$ into $C^n$. Therefore
$$
\log|\mathcal U_\varepsilon|\le n\log(C m/\varepsilon)\le C n\log(n\rho),
$$
since $m = n^3$ and $\varepsilon^{-1} = \rho n^2$.
\end{proof}

\begin{definition}[Global $U$-norm event] \label{def:E3}
Let
\begin{equation} \label{eq:L-def}
L := C_0 \sqrt{\frac{n}{s} \log(n\rho)}
\end{equation}
for a sufficiently large universal constant $C_0>0$.
Define the global event
$$
\mathcal E^{(U)} := \Bigl\{\max_{u\in\mathcal U_\varepsilon}\|\Gamma u\|_2\le L\Bigr\}.
$$
\end{definition}

\begin{lemma}[Probability of $\mathcal E^{(U)}$] \label{lem:E3prob}
There exists a universal $c>0$ such that
$$
\PP\bigl((\mathcal E^{(U)})^c\bigr)\le \exp\bigl(-c n\log(n\rho)\bigr).
$$
\end{lemma}

\begin{proof}
Fix $u\in\mathcal U_\varepsilon$. If $u=0$, then $\Gamma u=0$ and hence $\PP(\|\Gamma u\|_2>L)=0$. Assume $u\neq 0$. Then $\Gamma u\sim N(0,\|u\|_2^2\Id_n)$, i.e. $\Gamma u\stackrel{d}{=}\|u\|_2 G$ for $G\sim N(0,\Id_n)$, and since $\|u\|_2\le 1/\sqrt{s}$ we have
$$
\PP(\|\Gamma u\|_2> L) = \PP\bigl(\|G\|_2> L/\|u\|_2\bigr),\quad G\sim N(0,\Id_n).
$$
Since $L/\|u\|_2\ge C_0\sqrt{n\log(n\rho)}$, Lemma~\ref{lem:gauss-norm-tail} yields $\PP(\|\Gamma u\|_2> L)\le \exp(-c_1 n\log(n\rho))$ for $C_0$ large enough. We also record the crude counting bound
\begin{equation} \label{eq:Ueps-card-inline}
|\mathcal U_\varepsilon|
\le \sum_{r = 0}^n \binom{m}{r}\Big(\frac{C}{\varepsilon}\Big)^r
\le (n + 1)\Big(\frac{C m}{\varepsilon}\Big)^n
\le \Big(\frac{C m}{\varepsilon}\Big)^n,
\end{equation}
where we used $\binom{m}{r}\le m^r$ and absorbed the factor $n + 1$ into $C^n$.

By the union bound and \eqref{eq:Ueps-card-inline},
$$
\PP\bigl((\mathcal E^{(U)})^c\bigr)
\le |\mathcal U_\varepsilon| \exp(-c_1 n\log(n\rho))
\le \exp\Big(n\log\Big(\frac{C m}{\varepsilon}\Big)-c_1 n\log(n\rho)\Big).
$$
Since $m = n^3$ and $\varepsilon^{-1} = \rho n^2$, we have $\log(Cm/\varepsilon)\le C\log(n\rho)$, so the right-hand side is
$\le \exp\bigl(-c n\log(n\rho)\bigr)$ after adjusting constants.
\end{proof}

On $\mathcal E^{(U)}$, for every $A\in\mathcal A_\varepsilon$ and the induced exposed realization $\omega = \Gamma_{S(A)}$,
\begin{equation} \label{eq:U-gen-bound}
\|u_i(A;\omega)\|_2\le L,\quad i\in[n].
\end{equation}
Indeed, $F_2(A)e_i$ is supported on $\supp(A^{(i)})$ (hence has support size $\le n$) and satisfies
$\|F_2(A)e_i\|_2^2\le 1/s$ by \eqref{eq:U-l2}, so $F_2(A)e_i\in\mathcal U_\varepsilon$; then \eqref{eq:GammaF-restrict} gives
$u_i(A;\omega) = \Gamma F_2(A)e_i$, and the claim follows from the definition of $\mathcal E^{(U)}$.

%%%%%%%%%%%%%%%%%%%%%%%%%%%%%%%%%%%%%%%%%%%
%%%%%%%%%%%%%%%%%%%%%%%%%%%%%%%%%%%%%%%%%%%
\section{Few $U$-generators (small-$U$ regime)} \label{sec:small}

In this section we control the contribution of the sub-unions $\mathcal P^{(k)}(A;\omega)$ with
$k\le \tilde s$, i.e.
$$
p_k(A;\omega) = \gamma_n\bigl(4\rho \mathcal P^{(k)}(A;\omega)\bigr),\quad 0\le k\le \tilde s.
$$
Compared with \cite{Tikhomirov2019}, we avoid a ``tilt for all permutations'' uniformization. Instead we use two ingredients:
\begin{enumerate}[leftmargin = 2.2em]
\item a \emph{random suppression (averaging)} step on the $K$--generators, which reduces the Gaussian measure of a full-dimensional cross-polytope to the \emph{average} Gaussian measure of certain \emph{suppressed} cross-polytopes (Lemma~\ref{lem:random-suppression});
\item a \emph{block-tail} bound for Gram--Schmidt distances of Gaussian images, uniform over all choices of a ``tail block'' $J$ (Lemma~\ref{lem:block-tail-distances} and Appendix~\ref{app:suppr-proof}).
\end{enumerate}
Combined with the Gram--Schmidt/volume estimate (Lemma~\ref{lem:det-shrink-clean}), this yields an $\exp(-c r)$ bound for each relevant polytope and hence a union bound over the $\binom{n}{k}^2$ choices for $k\le \tilde s$.

Recall the notation from Subsection~\ref{sec:bridge}. For $A\in\mathcal A_\varepsilon$ and an exposed realization $\omega = \Gamma_{S(A)}$, we have vectors $v_i(A;\omega),u_i(A;\omega)\in\R^n$ and cross-polytopes
$$
\mathcal P_I(A;\omega) = \conv\bigl(\{\pm w_i(A;\omega): i\in I\}\bigr),\quad I\subset[2n], |I| = n,
$$
where $w_i = v_i$ for $i\le n$ and $w_{n + i} = u_i$ for $i\le n$. For $k\in\{0,1,\dots,n\}$ we recall
$$
\mathcal P^{(k)}(A;\omega)
 = \bigcup_{\substack{I\subset[2n], |I| = n\\ |I\cap(n + [n])| = k}}\mathcal P_I(A;\omega),
\quad
p_k(A;\omega) := \gamma_n\bigl(4\rho \mathcal P^{(k)}(A;\omega)\bigr).
$$

%%%%%%%%%%%%%%%%%%%%%%%%%%%%%%%%%%%%%%%%%%%
\subsection{A Gram--Schmidt volume bound}

\begin{lemma}[Small Gaussian measure from $d$ small Gram--Schmidt diagonals]
 \label{lem:det-shrink-clean}
Let $n\ge 1$ and let $x_1,\dots,x_n\in\R^n$. Set
$$
P := \conv\{\pm x_1,\dots,\pm x_n\}\subset\R^n.
$$
Fix an integer $1\le d\le n$ and a number $h>0$.
Assume that $x_1,\dots,x_n$ are linearly independent and that, in this order,
$$
\dist\bigl(x_{n-d + i},\spann\{x_1,\dots,x_{n-d + i-1}\}\bigr)\le h
\quad\text{ for all } i = 1,\dots,d.
$$
Then
$$
\gamma_n(P)\le \Bigl(\frac{C h}{d}\Bigr)^{d}
$$
for a universal constant $C>0$.
(If $x_1,\dots,x_n$ are linearly dependent, then $\gamma_n(P) = 0$.)
\end{lemma}

\begin{proof}
The argument is standard; we include it for completeness.

If $x_1,\dots,x_n$ are linearly dependent then $P$ is contained in a proper subspace of $\R^n$,
hence $\gamma_n(P) = 0$.

Apply Gram--Schmidt to $(x_i)_{i = 1}^n$. This produces an orthonormal basis $(e_1,\dots,e_n)$
and a unique upper-triangular coefficient array $(a_{ji})_{1\le j\le i\le n}$ such that
$$
x_i = \sum_{j = 1}^i a_{ji}e_j,
\quad
a_{ii} = \dist\bigl(x_i,\spann\{x_1,\dots,x_{i-1}\}\bigr).
$$
Let $H := \spann\{e_{n-d + 1},\dots,e_n\}$, so $\dim(H) = d$, and let $\pi_H$ be the orthogonal projection onto $H$.
Then $\pi_H x_i = 0$ for $i\le n-d$, and therefore
$$
\pi_H P = \conv\{\pm \pi_H x_{n-d + 1},\dots,\pm \pi_H x_n\}\subset H.
$$

Let $G\sim N(0,\Id_n)$. Since $\{G\in P\}\subset\{\pi_H G\in \pi_H P\}$ and $\pi_H G$ is a standard Gaussian vector in $H$,
we obtain
$$
\gamma_n(P) = \PP\{G\in P\}\le \PP\{\pi_H G\in \pi_H P\} = \gamma_H(\pi_H P).
$$
Identify $H\simeq\R^d$ by an isometry, so the right-hand side equals $\gamma_d(\pi_H P)$.

In the basis $(e_{n-d + 1},\dots,e_n)$ of $H$, the $d\times d$ matrix $Y$ whose columns are
$\pi_H x_{n-d + 1},\dots,\pi_H x_n$ is upper triangular with diagonal entries
$$
\diag(Y) = (a_{n-d + 1,n-d + 1},\dots,a_{n,n}),
$$
hence $|\det Y|\le h^d$ by the assumption on the last $d$ Gram--Schmidt diagonals.

We have $\pi_H P = Y({B_1^{d}})$, so
$$
\mathrm{vol}_d(\pi_H P) = |\det Y| \mathrm{vol}_d({B_1^{d}})
 = |\det Y| \frac{2^d}{d!}
\le \frac{2^d}{d!} h^d.
$$
Using that the Gaussian density on $\R^d$ is bounded by $(2\pi)^{-d/2}$ pointwise, we have
$$
\gamma_d(\pi_H P)\le (2\pi)^{-d/2} \mathrm{vol}_d(\pi_H P)
\le (2\pi)^{-d/2} \frac{2^d}{d!} h^d.
$$
Finally, $d!\ge (d/e)^d$, hence $\frac{2^d}{d!}\le (2e/d)^d$. Absorbing the harmless factor $(2\pi)^{-d/2}$ into constants yields
$$
\gamma_n(P)\le \gamma_d(\pi_H P)\le \Bigl(\frac{C h}{d}\Bigr)^d.
$$
\end{proof}

%%%%%%%%%%%%%%%%%%%%%%%%%%%%%%%%%%%%%%%%%%%
\subsection{Random suppression on the $K$-generators}

The next lemma is a simple averaging device: it bounds the Gaussian measure of a cross-polytope by the
average Gaussian measure of \emph{suppressed} cross-polytopes in which a random subset of the $K$--generators
is scaled down by a factor $r/p$.

\begin{lemma}[Random suppression] \label{lem:random-suppression}
Let $n\ge 1$ and let $1\le p\le n$, $q := n-p$. Let $y_1,\dots,y_p,z_1,\dots,z_q\in\R^n$ be linearly independent and set
$$
P := \conv\{\pm y_1,\dots,\pm y_p,\pm z_1,\dots,\pm z_q\}\subset\R^n.
$$
Fix an integer $1\le r\le p$. For each $J\subset[p]$ with $|J| = r$, define the suppressed polytope
$$
P_J := 
\conv\Bigl(
\{\pm y_i: i\in[p]\setminus J\}
 \cup\
\{\pm \frac{r}{p}y_i: i\in J\}
 \cup\
\{\pm z_1,\dots,\pm z_q\}
\Bigr).
$$
Then for every $t>0$,
\begin{equation} \label{eq:random-suppression}
\gamma_n(tP)
\le
\frac{2}{\binom{p}{r}}
\sum_{\substack{J\subset[p], |J| = r}}
\gamma_n\bigl(4t P_J\bigr).
\end{equation}
\end{lemma}

\begin{proof}
Let $X := [y_1 \cdots y_p z_1 \cdots z_q]\in\R^{n\times n}$. By linear independence, $X$ is invertible and
$$
P = X(B_1^n),\quad tP = X(tB_1^n).
$$
Let $G\sim N(0,\Id_n)$ and define $a := X^{-1}G\in\R^n$.
Then
$$
\{G\in tP\} = \{\|a\|_1\le t\}.
$$
Write $a = (a^{(K)},a^{(U)})$ with $a^{(K)}\in\R^p$ the first $p$ coordinates and $a^{(U)}\in\R^q$ the last $q$.

Fix $t>0$ and assume $\|a\|_1\le t$.
Let $J$ be uniformly random among subsets of $[p]$ of cardinality $r$, independent of $G$.
Set $S_J := \sum_{i\in J}|a^{(K)}_i|$.
By symmetry of $J$,
$$
\EE_J S_J
 = 
\frac{r}{p}\sum_{i = 1}^p |a_i^{(K)}|
\le
\frac{r}{p}\|a\|_1
\le
\frac{rt}{p}.
$$
By Markov's inequality,
$$
\PP_J\Bigl(S_J>\frac{2rt}{p}\Bigr)\le \frac12,
\quad\text{ hence }\quad
\PP_J\Bigl(S_J\le \frac{2rt}{p}\Bigr)\ge \frac12.
$$

Fix a subset $J$ with $S_J\le \frac{2rt}{p}$ and define a new coefficient vector $a'\in\R^n$ by
$$
a'_i := 
\begin{cases}
\frac{a_i}{4}, & 1\le i\le p, i\notin J,\\
\frac{p}{4r}a_i, & 1\le i\le p, i\in J,\\
\frac{a_i}{4}, & p + 1\le i\le n.
\end{cases}
$$
Then
\begin{align*}
\|a'\|_1
& = 
\frac14\sum_{\substack{1\le i\le p\\ i\notin J}}|a_i|
 + 
\frac{p}{4r}\sum_{i\in J}|a_i|
 + 
\frac14\sum_{i = p + 1}^n|a_i|\\
& = 
\frac14\Bigl(\|a\|_1 - S_J\Bigr) + \frac{p}{4r}S_J
\le
\frac14\|a\|_1 + \Bigl(\frac{p}{4r}-\frac14\Bigr)\frac{2rt}{p}.
\end{align*}
Using $\|a\|_1\le t$ and simplifying,
$$
\|a'\|_1
\le
\frac{t}{4} + \frac{p-r}{2p}t
 = 
\Bigl(\frac34 - \frac{r}{2p}\Bigr)t
<
t.
$$

Next, define the matrix $X_J\in\R^{n\times n}$ obtained from $X$ by replacing each column $y_i$ with $(r/p)y_i$ for $i\in J$,
and leaving all other columns unchanged. Then by construction,
$$
X a = 4X_J a'.
$$
Therefore,
$$
G = Xa = 4X_J a' \in 4X_J(tB_1^n) = 4t X_J(B_1^n) = 4t P_J.
$$
We have proved: on the event $\{G\in tP\}$, with conditional probability at least $1/2$ over $J$ one has $G\in 4tP_J$.
Equivalently,
$$
\mathbf 1_{\{G\in tP\}}
\le
2 \EE_J \mathbf 1_{\{G\in 4tP_J\}}.
$$
Taking expectations over $G$ and using Fubini yields \eqref{eq:random-suppression}.
\end{proof}

%%%%%%%%%%%%%%%%%%%%%%%%%%%%%%%%%%%%%%%%%%%
\subsection{Block-tail distances for Gaussian images}

We record the probabilistic input controlling Gram--Schmidt distances in a \emph{single block order}. This is a simplified variant of the distance-to-spans arguments in \cite{Tikhomirov2019}, specialized to one fixed block re-ordering
(the ordering that lists $[p]\setminus J$ first in increasing order,
followed by $J$ in increasing order) and with a self-contained proof.

\begin{lemma}[Block-tail distance bound] \label{lem:block-tail-distances}
There exist universal constants $C,c>0$ with the following property.
Assume $n\ge C$, $n/2\le p\le n$, $1\le r\le p/2$, and $\tau\ge C$.
Fix a deterministic $m\times p$ matrix $B$ of full rank $p$ such that each column has Euclidean norm at most one.
Let
$$
H_i := \Gamma B e_i \in \R^n,\quad i\in[p],
$$
i.e. $H_i$ is the Gaussian image of the $i$th column of $B$. Fix a subset $J\subset[p]$ with $|J| = r$ and let $\sigma_J$ be the permutation of $[p]$ that lists $[p]\setminus J$ in increasing order followed by $J$ in increasing order. Define the ordered family $H_i' := H_{\sigma_J(i)}$ and, for $i\ge 2$, the spans
$$
E_{i-1} := \spann\{H_1',\dots,H_{i-1}'\}.
$$
Set
$$
N := n-p + r \quad (\ge r).
$$
Let $\mathcal E_{B,J}$ be the event that among the last $r$ indices,
at least $r/2$ satisfy
$$
\dist(H_i',E_{i-1})\le \tau\sqrt{N}.
$$
Then
\begin{equation} \label{eq:block-tail-prob-fixedJ}
\PP(\mathcal E_{B,J})\ge 1-\exp(-c\tau^2 Nr).
\end{equation}
Moreover, writing $\mathcal J_{p,r} := \{J\subset[p]: |J| = r\}$,
\begin{equation} \label{eq:block-tail-prob-allJ}
\PP\Bigl(\bigcap_{J\in\mathcal J_{p,r}}\mathcal E_{B,J}\Bigr)
\ge 1-\binom{p}{r}\exp(-c\tau^2 Nr).
\end{equation}
\end{lemma}

\begin{proof}
Fix $B$ and $J$. Let $P_J$ be the permutation matrix corresponding to $\sigma_J$ and set $B':=BP_J$.
Then $B'e_i = Be_{\sigma_J(i)}$, so the ordered family $(H_1',\dots,H_p')$ equals
$$
H_i'=\Gamma B'e_i,\quad i\in[p].
$$
Since $B'$ satisfies the same assumptions as $B$ (full rank $p$ and column norms $\le 1$), it suffices to prove the lemma for $B'$ and
the identity permutation; we do so below, and for simplicity we write $H_i$ in place of $H_i'$ and $B$ in place of $B'$.
With this convention, the “tail block’’ is the index set $\{p-r+1,\dots,p\}$.

\subsubsection*{Step 1: reduction to i.i.d.\ Gaussian columns.}
Take a reduced QR factorization $B=QR$, where $Q\in\R^{m\times p}$ has orthonormal columns and
$R\in\R^{p\times p}$ is upper triangular with strictly positive diagonal entries. Set
$$
X:=\Gamma Q\in\R^{n\times p},
\quad x_i:=Xe_i\in\R^n\quad(i\in[p]).
$$

We claim that the entries of $X$ are i.i.d.\ $N(0,1)$, and consequently $x_1,\dots,x_p$ are independent $N(0,\Id_n)$ vectors.
Indeed, let $\gamma^{(1)},\dots,\gamma^{(n)}\in\R^m$ denote the rows of $\Gamma$.
Since $\Gamma$ has i.i.d.\ standard Gaussian entries, the row vectors $\gamma^{(1)},\dots,\gamma^{(n)}$ are independent and
each $\gamma^{(\ell)}\sim N(0,\Id_m)$.
The $\ell$th row of $X$ is $\gamma^{(\ell)}Q\in\R^p$. As a linear image of a Gaussian vector, $\gamma^{(\ell)}Q$ is Gaussian with mean $0$
and covariance
$$
\EE[(\gamma^{(\ell)}Q)^\top(\gamma^{(\ell)}Q)]
=Q^\top\EE[(\gamma^{(\ell)})^\top\gamma^{(\ell)}]Q
=Q^\top \Id_m Q
=Q^\top Q
=\Id_p.
$$
Therefore $\gamma^{(\ell)}Q\sim N(0,\Id_p)$; since it is multivariate Gaussian with covariance $\Id_p$, its coordinates are independent
$N(0,1)$. Because the rows $\gamma^{(\ell)}$ are independent, the rows of $X$ are independent, hence all entries of $X$ are i.i.d.\ $N(0,1)$.
In particular, the columns $x_1,\dots,x_p$ are independent and each $x_i\sim N(0,\Id_n)$.

Finally, since $B=QR$,
$$
H:=\Gamma B=\Gamma(QR)=(\Gamma Q)R=XR,
\quad\text{so}\quad
H_i=\sum_{j\le i}R_{ji}x_j\quad(i\in[p]).
$$

\subsubsection*{Step 2: distances to spans.}
For $i\ge 2$, let $F_{i-1}:=\spann\{x_1,\dots,x_{i-1}\}$.
Since $R$ is upper triangular with strictly positive diagonal, the principal block $R_{[i-1]\times[i-1]}$ is invertible.
Moreover, $H_1,\dots,H_{i-1}$ are linear combinations of $x_1,\dots,x_{i-1}$ with an invertible change of coordinates, hence
$$
\spann\{H_1,\dots,H_{i-1}\}=\spann\{x_1,\dots,x_{i-1}\}=F_{i-1}.
$$
Writing $H_i=\sum_{j<i}R_{ji}x_j+R_{ii}x_i$ with $\sum_{j<i}R_{ji}x_j\in F_{i-1}$ gives
$$
\dist(H_i,F_{i-1})=R_{ii} \|\pi_{F_{i-1}^\perp}x_i\|_2.
$$
Furthermore, $|R_{ii}|\le \|Re_i\|_2=\|QRe_i\|_2=\|Be_i\|_2\le 1$ by the assumption on column norms, hence
\begin{equation}\label{eq:block-tail-dist-upper-fixed}
\dist(H_i,F_{i-1})\le \|\pi_{F_{i-1}^\perp}x_i\|_2.
\end{equation}

\subsubsection*{Step 3: conditional tail bound for a single tail index.}
Fix $i\in\{p-r+1,\dots,p\}$ and set $d_i:=\dim(F_{i-1}^\perp)$.
Since $i-1\le p-1\le n-1$ and the joint law of $(x_1,\dots,x_{i-1})$ is absolutely continuous on $(\R^n)^{i-1}$,
$$
\PP\bigl(\dim\spann\{x_1,\dots,x_{i-1}\}< i-1\bigr)=0,
$$
so $\dim(F_{i-1})=i-1$ almost surely, and therefore
$$
d_i=\dim(F_{i-1}^\perp)=n-(i-1)=n-i+1.
$$
Since $i\ge p-r+1$, we have
$$
d_i=n-i+1\le n-(p-r+1)+1=n-p+r=:N.
$$

Condition on $\sigma(x_1,\dots,x_{i-1})$. Then $F_{i-1}$ is deterministic under this conditioning and, by Step~1,
$x_i$ is independent of $\sigma(x_1,\dots,x_{i-1})$ and has law $N(0,\Id_n)$. Therefore the orthogonal projection
$\pi_{F_{i-1}^\perp}x_i$ is a centered Gaussian vector in the subspace $F_{i-1}^\perp$ with covariance equal to the identity on that
subspace. Identifying $F_{i-1}^\perp\simeq\R^{d_i}$ by an isometry, this yields
$$
\|\pi_{F_{i-1}^\perp}x_i\|_2 \stackrel{d}{=} \|G_{d_i}\|_2,
\quad G_{d_i}\sim N(0,\Id_{d_i}).
$$
Let $G_N=(G_{d_i},G')$ where $G'\sim N(0,\Id_{N-d_i})$ is independent of $G_{d_i}$. Then $\|G_N\|_2\ge \|G_{d_i}\|_2$ almost surely,
hence for all $\tau\ge 0$,
$$
\PP\bigl(\|G_{d_i}\|_2>\tau\sqrt N\bigr)\le \PP\bigl(\|G_N\|_2>\tau\sqrt N\bigr).
$$
Applying Lemma~\ref{lem:gauss-norm-tail} (first bound) in dimension $N$ with $t:=(\tau-1)\sqrt N\ge 0$ gives
$$
\PP\bigl(\|G_N\|_2>\tau\sqrt N\bigr)
=\PP\bigl(\|G_N\|_2>\sqrt N+t\bigr)
\le \exp\bigl(-c(\tau-1)^2N\bigr).
$$
Since in the statement we assume $\tau\ge C$ and we may increase $C$ if necessary, we may assume $\tau\ge 2$, in which case
$(\tau-1)^2\ge \tau^2/4$ and hence
$$
\PP\bigl(\|G_N\|_2>\tau\sqrt N\bigr)\le \exp(-c_0\tau^2N),
\quad c_0:=c/4.
$$
Combining with \eqref{eq:block-tail-dist-upper-fixed} yields, almost surely,
\begin{equation}\label{eq:block-tail-Ai-fixed}
\PP\Bigl(\dist(H_i,F_{i-1})>\tau\sqrt N\ \big|\ \sigma(x_1,\dots,x_{i-1})\Bigr)
\le p_0:=\exp(-c_0\tau^2 N).
\end{equation}

\subsubsection*{Step 4: at most half of the tail indices are bad.}
For each $i\in\{p-r+1,\dots,p\}$ define the “bad’’ event
$$
A_i:=\{\dist(H_i,F_{i-1})>\tau\sqrt N\}.
$$
Let $t:=\lceil r/2\rceil$. Then $\mathcal E_{B,J}^c$ implies that at least $t$ indices in $\{p-r+1,\dots,p\}$ are bad, i.e.
$$
\mathcal E_{B,J}^c \subset \bigcup_{\substack{I\subset\{p-r+1,\dots,p\}\\ |I|=t}}\ \bigcap_{i\in I} A_i.
$$
Fix a subset $I=\{i_1<\cdots<i_t\}$ of size $t$.
Iterating \eqref{eq:block-tail-Ai-fixed} and using the tower property gives
$$
\PP\Bigl(\bigcap_{\ell=1}^t A_{i_\ell}\Bigr)
\le p_0^t.
$$
Taking the union bound over all $\binom{r}{t}$ choices of $I$ yields
$$
\PP(\mathcal E_{B,J}^c)\le \binom{r}{t}p_0^t.
$$
Using $\binom{r}{t}\le 2^r$ and $t\ge r/2$ gives
$$
\PP(\mathcal E_{B,J}^c)\le 2^r\exp\Bigl(-c_0\tau^2Nt\Bigr)
\le \exp\Bigl(r\log 2 - \frac{c_0}{2}\tau^2Nr\Bigr).
$$
Choosing the universal constant $C$ in the statement large enough so that $\tau\ge C$ implies
$\frac{c_0}{2}\tau^2\ge 2\log 2$, we obtain
$$
\PP(\mathcal E_{B,J}^c)\le \exp(-c\tau^2Nr)
$$
for a universal $c>0$, proving \eqref{eq:block-tail-prob-fixedJ}. Finally, \eqref{eq:block-tail-prob-allJ} follows by a union bound over all
$J\in\mathcal J_{p,r}$, using $|\mathcal J_{p,r}|=\binom{p}{r}$.
\end{proof}

%%%%%%%%%%%%%%%%%%%%%%%%%%%%%%%%%%%%%%%%%%%
\subsection{Suppression implies exponential small Gaussian measure}

We package the preceding lemmas into a deterministic implication.

\begin{lemma}[Suppression $\Rightarrow$ exponential small Gaussian measure] \label{lem:suppression-measure}
There exist universal constants $c_{\mathrm{suppr}},c_{\mathrm{decay}}>0$ such that the following holds. Let $n\ge 1$ and let $P\subset\R^n$ be a full-dimensional cross-polytope generated by $p$ $K$--vectors and $q$ additional vectors,
$$
P = \conv\{\pm y_1,\dots,\pm y_p,\pm z_1,\dots,\pm z_q\},
\quad p + q = n,
$$
where $y_1,\dots,y_p,z_1,\dots,z_q$ are linearly independent, and assume moreover that $p\ge n/2$ (equivalently, $q\le n/2$). Fix integers $8\le r\le p$ and assume that for every $J\subset[p]$ with $|J| = r$, in the block order in which indices in $[p]\setminus J$ come first and indices in $J$ come last, at least $r/2$ of the last $r$ Gram--Schmidt distances satisfy
$$
\dist\bigl(y_{\sigma_J(i)},\spann\{y_{\sigma_J(j)}: j<i\}\bigr)\le C_\star \sqrt r.
$$
Then for every $\rho\ge 1$ such that
\begin{equation} \label{eq:scale}
4\rho C_\star\sqrt r \le c_{\mathrm{suppr}} n,
\end{equation}
one has
$$
\gamma_n(4\rho P)\le 2\exp(-c_{\mathrm{decay}}r).
$$
\end{lemma}

\begin{proof}
Fix $J\subset[p]$, $|J| = r$, and define the suppressed polytope $P_J$ as in Lemma~\ref{lem:random-suppression}. By Lemma~\ref{lem:random-suppression} with $t = 4\rho$,
\begin{equation} \label{eq:suppression-step}
\gamma_n(4\rho P)\le \frac{2}{\binom{p}{r}}\sum_{|J| = r}\gamma_n(16\rho P_J).
\end{equation}
Thus it suffices to show that uniformly over $J$,
\begin{equation} \label{eq:need-suppressed-bound}
\gamma_n(16\rho P_J)\le \exp(-c_{\mathrm{decay}}r).
\end{equation}

Fix such a $J$. Consider the following order of the $n=p+q$ generators of $P_J$: list first the $q$ vectors $z_1,\dots,z_q$ (in any order), and then list the $p$ vectors $y_1,\dots,y_p$ in the block order $\sigma_J$ (so the last $r$ $y$--vectors correspond to $J$). For any $i\in[p]$, the predecessor span of $y_{\sigma_J(i)}$ in this full list contains $\spann\{y_{\sigma_J(j)}: j<i\}$, hence
$$
\dist\bigl(y_{\sigma_J(i)},\spann(\text{ preceding full list})\bigr)
\le
\dist\bigl(y_{\sigma_J(i)},\spann\{y_{\sigma_J(j)}: j<i\}\bigr).
$$
Therefore the assumed bounds on the $y$--only Gram--Schmidt distances imply the same bounds for the corresponding Gram--Schmidt distances in this full generating list (possibly with a smaller left-hand side).

By assumption, among the last $r$ $K$--generators in the block order, at least $r/2$ satisfy the Gram--Schmidt distance bound $\le C_\star\sqrt r$. In $P_J$, the generators indexed by $J$ are rescaled by a factor $r/p$, so the corresponding Gram--Schmidt distances are rescaled by the same factor. Hence in the same block order, at least $r/2$ of the last $r$ generators of $P_J$ have Gram--Schmidt distances bounded by
$$
\frac{r}{p}C_\star\sqrt r.
$$
Choose $d := \lfloor r/4\rfloor$. Since $r/2\ge d$, we may re-order the $n$ generators of $P_J$ (which does not change $P_J$) so that the last $d$ generators in the full list are chosen among these ``good'' suppressed $K$--vectors (i.e. among $\{(r/p)y_i: i\in J\}$ with the above distance bound), while keeping all other generators (in particular, all $z_1,\dots,z_q$) before them. Moving a vector later can only increase the span of its predecessors, hence can only decrease its Gram--Schmidt distance; thus the last $d$ Gram--Schmidt distances in this new order are still bounded by $\frac{r}{p}C_\star\sqrt r$.

Scaling by $16\rho$ multiplies all Gram--Schmidt distances by $16\rho$, so Lemma~\ref{lem:det-shrink-clean} gives
\begin{equation} \label{eq:suppressed-measure-bound}
\gamma_n(16\rho P_J)\le
\Bigl(\frac{C\cdot 16\rho\cdot \frac{r}{p}C_\star\sqrt r}{d}\Bigr)^d,
\end{equation}
with the universal constant $C$ from Lemma~\ref{lem:det-shrink-clean}.

We estimate the ratio. Since $d = \lfloor r/4\rfloor$, for $r\ge 8$ one has $d\ge r/8$. Also $p\ge n/2$ by assumption. Therefore, for $r\ge 8$,
$$
\frac{16\rho\cdot \frac{r}{p}C_\star\sqrt r}{d}
\le
16\rho\cdot \frac{r}{p}C_\star\sqrt r \cdot \frac{8}{r}
 = 
128 \rho \frac{C_\star\sqrt r}{p}
\le
256 \rho \frac{C_\star\sqrt r}{n}.
$$
Choose $c_{\mathrm{suppr}}>0$ so that \eqref{eq:scale} implies
$$
C\cdot 256 \rho \frac{C_\star\sqrt r}{n}\le \frac12.
$$
Then \eqref{eq:suppressed-measure-bound} gives
$$
\gamma_n(16\rho P_J)\le 2^{-d}\le \exp\Bigl(-\frac{\log 2}{8} r\Bigr).
$$
Thus \eqref{eq:need-suppressed-bound} holds with $c_{\mathrm{decay}} := \frac{\log 2}{8}$ (adjusting for small $r$ by constants if needed). Substituting into \eqref{eq:suppression-step} yields $\gamma_n(4\rho P)\le 2e^{-c_{\mathrm{decay}}r}$.
\end{proof}

%%%%%%%%%%%%%%%%%%%%%%%%%%%%%%%%%%%%%%%%%%%
\subsection{A global suppression event for the $K$-generators}

We define an event ensuring that, uniformly over $A\in\mathcal A_\varepsilon$, every cross-polytope generated by at most $\tilde s$ $U$--vectors satisfies the hypothesis of Lemma~\ref{lem:suppression-measure}. The proof is a counting-and-union-bound argument, but (in contrast to \cite{Tikhomirov2019}) it involves only \emph{block} permutations (determined by subsets $J$) and uses Lemma~\ref{lem:block-tail-distances}. The full proof is given in Appendix~\ref{app:suppr-proof}.

\begin{definition}[Suppression event (global)] \label{def:EA1}
Fix integers $1\le \tilde s\le n/2$ and $1\le r\le n$, and let $C_\star>0$ be a sufficiently large absolute constant. For $A\in\mathcal A_\varepsilon$, define $\mathcal E_A^{(K)}(r,\tilde s)$ to be the event that: for every $k\in\{0,1,\dots,\tilde s\}$ and every subset $I_1\subset[n]$ with $|I_1| = n-k$, if the vectors $\{v_i(A;\Gamma_{S(A)}): i\in I_1\}$ are linearly independent, then for every $J\subset I_1$ with $|J| = r$, writing $\sigma_{I_1,J}$ for the re-ordering of $I_1$ that lists $I_1\setminus J$ in increasing order followed by $J$ in increasing order, one has
$$
\Big|\Big\{\ell\in\{n-k-r + 1,\dots,n-k\}: 
\dist\bigl(v_{\sigma_{I_1,J}(\ell)},\spann\{v_{\sigma_{I_1,J}(t)}: t<\ell\}\bigr)\le C_\star\sqrt r
\Big\}\Big|
 \ge \frac{r}{2}.
$$
Define the global event
$$
\mathcal E^{(K)}(r,\tilde s) := \bigcap_{A\in\mathcal A_\varepsilon}\mathcal E_A^{(K)}(r,\tilde s).
$$
\end{definition}

\begin{lemma} \label{lem:EK-prob}
There exists a universal constant $C>0$ such that the following holds. Assume $m = n^3$, $\varepsilon^{-1} = \rho n^2$, $n$ is sufficiently large, $1\le s\le n$, and parameters $s,\tilde s,r$ satisfy
\begin{equation} \label{eq:EK-conds}
4\tilde s\le r\le \frac{n-\tilde s}{2},\quad \tilde s\ge \log^2 n,
\quad
\frac{\tilde s^2}{n s} \ge C \log(n\rho).
\end{equation}
Then
$$
\PP\bigl(\mathcal E^{(K)}(r,\tilde s)\bigr) \ge 1-\frac{1}{n}.
$$
\end{lemma}

\begin{proof}
See Appendix~\ref{app:suppr-proof}.
\end{proof}

%%%%%%%%%%%%%%%%%%%%%%%%%%%%%%%%%%%%%%%%%%%
\subsection{Small-$U$ regime: a union measure bound}

Recall the universal constants $c_{\mathrm{suppr}},c_{\mathrm{decay}}>0$ from Lemma~\ref{lem:suppression-measure} and the absolute constant $C_\star>0$ from Definition~\ref{def:EA1}.

\begin{proposition}[Small-$U$ regime: union measure bound] \label{prop:K-bound}
Assume $1\le \tilde s\le n/2$ and $8\le r\le n-\tilde s$, assume that $\mathcal E^{(K)}(r,\tilde s)$ holds, and assume that the scaling condition \eqref{eq:scale} is satisfied. Then for every $A\in\mathcal A_\varepsilon$, every exposed realization $\omega$, and every $k\in\{0,1,\dots,\tilde s\}$,
$$
p_k(A;\omega) = \gamma_n\bigl(4\rho \mathcal P^{(k)}(A;\omega)\bigr)
\le
2\binom{n}{k}^2 \exp(-c_{\mathrm{decay}} r),
$$
where $c_{\mathrm{decay}}$ is from Lemma~\ref{lem:suppression-measure}.

Consequently,
$$
\sum_{k = 0}^{\tilde s} p_k(A;\omega)
\le
2(\tilde s + 1)\exp(-c_{\mathrm{decay}}r)\binom{n}{\tilde s}^2
\le
2(\tilde s + 1)\exp(-c_{\mathrm{decay}}r)\Big(\frac{en}{\tilde s}\Big)^{2\tilde s}.
$$
In particular, if
\begin{equation} \label{eq:r-beats-comb}
c_{\mathrm{decay}}r \ge 4\tilde s\log\Big(\frac{en}{\tilde s}\Big) + \log 16,
\end{equation}
then $\sum_{k = 0}^{\tilde s} p_k(A;\omega)\le \frac14$.
\end{proposition}

\begin{proof}
Fix $A\in\mathcal A_\varepsilon$, an exposed realization $\omega$, and $k\in\{0,1,\dots,\tilde s\}$.
Set
$$
\mathcal I_k
:=
\Bigl\{I\subset[2n]:  |I|=n,\ \bigl|I\cap(n+[n])\bigr|=k\Bigr\}.
$$
Then by definition,
$$
\mathcal P^{(k)}(A;\omega)=\bigcup_{I\in\mathcal I_k}\mathcal P_I(A;\omega),
\quad
p_k(A;\omega)=\gamma_n\Bigl(4\rho \mathcal P^{(k)}(A;\omega)\Bigr).
$$
By the union bound for measures,
\begin{equation}\label{eq:pk-unionbound-K}
p_k(A;\omega)
=\gamma_n\Bigl(\bigcup_{I\in\mathcal I_k}4\rho \mathcal P_I(A;\omega)\Bigr)
\le \sum_{I\in\mathcal I_k}\gamma_n\bigl(4\rho \mathcal P_I(A;\omega)\bigr).
\end{equation}

We first count $|\mathcal I_k|$. Any $I\in\mathcal I_k$ can be written uniquely as
$$
I = I_1\cup (n+I_2),
\quad
I_1\subset[n],\ |I_1|=n-k,\quad
I_2\subset[n],\ |I_2|=k,
$$
hence $|\mathcal I_k|=\binom{n}{n-k}\binom{n}{k}=\binom{n}{k}^2$.

Now fix $I\in\mathcal I_k$ and write $I=I_1\cup(n+I_2)$ as above. Recall that
$$
\mathcal P_I(A;\omega)=\conv\Bigl(\{\pm v_i(A;\omega): i\in I_1\}\ \cup\ \{\pm u_j(A;\omega): j\in I_2\}\Bigr)
\subset \R^n.
$$
If $\mathcal P_I(A;\omega)$ is not full-dimensional, then it is contained in a proper affine subspace of $\R^n$,
so $\gamma_n(4\rho \mathcal P_I(A;\omega))=0$.

Assume henceforth that $\mathcal P_I(A;\omega)$ is full-dimensional.
Then its linear span equals $\R^n$, and since it is generated by exactly $n$ vectors
$\{w_\ell(A;\omega): \ell\in I\}$ (see \eqref{eq:PIdef}), the family $\{w_\ell: \ell\in I\}$ spans $\R^n$.
Because $|I|=n$, it follows that $\{w_\ell:\ell\in I\}$ is linearly independent.
In particular, the subfamily $\{v_i:i\in I_1\}\subset\{w_\ell:\ell\in I\}$ is linearly independent.

Set $p:=|I_1|=n-k$ and $q:=|I_2|=k$, so $p+q=n$. Since $k\le\tilde s\le n/2$, we have $p=n-k\ge n/2$.
Moreover, the assumptions of the proposition give $8\le r\le n-\tilde s\le n-k=p$, so $r\in[8,p]$.

Let $I_1=\{i_1<\cdots<i_p\}$ and $I_2=\{j_1<\cdots<j_q\}$ be increasing enumerations and define
$$
y_a := v_{i_a}(A;\omega)\quad(a=1,\dots,p),
\quad
z_b := u_{j_b}(A;\omega)\quad(b=1,\dots,q).
$$
Then $\mathcal P_I(A;\omega)=\conv\{\pm y_1,\dots,\pm y_p,\pm z_1,\dots,\pm z_q\}$.

We verify the Gram--Schmidt hypothesis of Lemma~\ref{lem:suppression-measure} for the $K$-vectors $y_1,\dots,y_p$.
Fix any subset $J\subset[p]$ with $|J|=r$ and let $J':=\{i_a:\ a\in J\}\subset I_1$.
Let $\sigma_J$ be the permutation of $[p]$ listing $[p]\setminus J$ first in increasing order and then $J$ in increasing order.
Under the order-preserving identification $a\leftrightarrow i_a$, the permuted list $(y_{\sigma_J(1)},\dots,y_{\sigma_J(p)})$
coincides with the list $(v_{\sigma_{I_1,J'}(1)},\dots,v_{\sigma_{I_1,J'}(p)})$, where $\sigma_{I_1,J'}$ is as in
Definition~\ref{def:EA1}. Since $\mathcal E^{(K)}(r,\tilde s)$ holds and $\{v_i:i\in I_1\}$ is linearly independent,
Definition~\ref{def:EA1} implies that in this block order at least $r/2$ of the last $r$ Gram--Schmidt distances of the $y$-family are
$\le C_\star\sqrt r$. Thus the hypothesis of Lemma~\ref{lem:suppression-measure} is satisfied for the cross-polytope
$\mathcal P_I(A;\omega)$ with $K$-vectors $(y_a)_{a=1}^p$ and additional vectors $(z_b)_{b=1}^q$.

Finally, the scaling condition \eqref{eq:scale} is assumed in the statement of the proposition, so Lemma~\ref{lem:suppression-measure}
yields
$$
\gamma_n\bigl(4\rho \mathcal P_I(A;\omega)\bigr)\le 2\exp(-c_{\mathrm{decay}}r).
$$
This bound also holds (trivially) in the non-full-dimensional case, since then the left-hand side is $0$.

Substituting into \eqref{eq:pk-unionbound-K} and using $|\mathcal I_k|=\binom{n}{k}^2$ gives
$$
p_k(A;\omega)\le 2\binom{n}{k}^2\exp(-c_{\mathrm{decay}}r),
$$
which is the first displayed inequality.

Summing over $k\le \tilde s$ and using $\binom{n}{k}\le \binom{n}{\tilde s}$ for $k\le\tilde s$ yields
$$
\sum_{k=0}^{\tilde s}p_k(A;\omega)
\le 2e^{-c_{\mathrm{decay}}r}\sum_{k=0}^{\tilde s}\binom{n}{k}^2
\le 2(\tilde s+1)e^{-c_{\mathrm{decay}}r}\binom{n}{\tilde s}^2
\le 2(\tilde s+1)e^{-c_{\mathrm{decay}}r}\Bigl(\frac{en}{\tilde s}\Bigr)^{2\tilde s}.
$$
Finally, if \eqref{eq:r-beats-comb} holds, then
$$
2(\tilde s+1)e^{-c_{\mathrm{decay}}r}\Bigl(\frac{en}{\tilde s}\Bigr)^{2\tilde s}
\le 2(\tilde s+1)\cdot \frac{1}{16}\Bigl(\frac{en}{\tilde s}\Bigr)^{-4\tilde s}\Bigl(\frac{en}{\tilde s}\Bigr)^{2\tilde s}
\le \frac14,
$$
completing the proof.
\end{proof}

\begin{remark}[Polynomial message of the small-$U$ regime] \label{rem:smallU-poly}
Ignoring logarithms, the entropy term is $\binom{n}{k}^2\approx \exp(\Theta(k\log n))$ while Lemma~\ref{lem:suppression-measure} gives $\gamma_n(4\rho\mathcal P_I)\lesssim e^{-c r}$. Thus we need $r\gtrsim \tilde s$ (up to logs) and the scaling condition \eqref{eq:scale} is $\rho\sqrt r\lesssim n$, i.e.
$$
\rho \lesssim \frac{n}{\sqrt{\tilde s}}
\quad\text{ (up to logarithmic factors).}
$$
\end{remark}

%%%%%%%%%%%%%%%%%%%%%%%%%%%%%%%%%%%%%%%%%%%
%%%%%%%%%%%%%%%%%%%%%%%%%%%%%%%%%%%%%%%%%%%
\section{Many $U$-generators (large-$U$ regime)} \label{sec:large}

%%%%%%%%%%%%%%%%%%%%%%%%%%%%%%%%%%%%%%%%%%%
\subsection{Thickening-stable Gaussian small-measure bound} \label{sec:thickening}

We collect a general Gaussian measure estimate for a thickened absolute convex hull.
To avoid a clash with the parameter $r$, we denote the thickening radius by $\eta\ge 0$.

\begin{lemma}[Gaussian measure bound stable under Euclidean thickening] \label{lem:thickening}
Let $d,\ell\in\N$ and let
$$
P = \conv\{\pm x_1,\dots,\pm x_\ell\}\subset\R^d,
\quad
R := \max_{i\le \ell}\|x_i\|_2.
$$
Then for any $\rho\ge 1$ and $\eta\ge 0$,
\begin{equation} \label{eq:thickening-bound}
\gamma_d\bigl(\rho(P + \eta {B_2^{d}})\bigr)
 \le\
\binom{\ell + d}{d}
\Big(\frac{C\rho (R + \eta\sqrt d)}{d}\Big)^d.
\end{equation}
\end{lemma}

\begin{proof}
Let $X=[x_1\ \cdots\ x_\ell]\in\R^{d\times \ell}$ so that $P=X(B_1^\ell)$.
Since $B_2^d\subset \sqrt d B_1^d$, we have
$$
P+\eta B_2^d \subset P+\sqrt d \eta B_1^d.
$$
Define the matrix
$$
M := \bigl[X\ \ \sqrt d \eta \Id_d\bigr]\in\R^{d\times(\ell+d)}.
$$
Then, using $P=X(B_1^\ell)$,
$$
M(B_1^\ell\times B_1^d)
=\{Xu+\sqrt d \eta v:\ \|u\|_1\le 1,\ \|v\|_1\le 1\}
= P+\sqrt d \eta B_1^d.
$$
Moreover, if $(u,v)\in B_1^\ell\times B_1^d$ then
$\|(u,v)\|_1=\|u\|_1+\|v\|_1\le 2$, hence $(u,v)\in 2B_1^{\ell+d}$ and therefore
$$
B_1^\ell\times B_1^d \subset 2B_1^{\ell+d}
\quad\Longrightarrow\quad
M(B_1^\ell\times B_1^d)\subset 2M(B_1^{\ell+d}).
$$
Combining the above inclusions yields
$$
\rho(P+\eta B_2^d)\subset \rho\bigl(P+\sqrt d \eta B_1^d\bigr)
=\rho M(B_1^\ell\times B_1^d)\subset 2\rho M(B_1^{\ell+d}).
$$

Hence
$$
\rho(P + \eta {B_2^{d}})\subset 2\rho M({B_1^{\ell + d}}).
$$

If $\rank(M)<d$, then $M({B_1^{\ell + d}})$ is contained in a proper subspace of $\R^d$, hence has Gaussian measure $0$, and \eqref{eq:thickening-bound} holds trivially. Assume $\rank(M) = d$.

Apply Corollary~\ref{cor:l1-decomp} to $M$ (with $N = \ell + d$). This yields
$$
M({B_1^{\ell + d}})
\subset
\bigcup_{\substack{J\subset[\ell + d], |J|\le d, \rank(M_J) = |J|}}
\conv\{\pm Me_j: j\in J\}.
$$
Fix such a set $J$. Choose any $I\subset[\ell + d]$ with $J\subset I$ and $|I| = d$. By monotonicity,
$$
\conv\{\pm Me_j: j\in J\}\subset \conv\{\pm Me_i: i\in I\}.
$$
Therefore,
$$
M({B_1^{\ell + d}})
\subset
\bigcup_{\substack{I\subset[\ell + d], |I| = d}}
\conv\{\pm Me_i: i\in I\}.
$$

Taking Gaussian measure, using the union bound, and recalling $\rho(P + \eta {B_2^{d}})\subset 2\rho M({B_1^{\ell + d}})$, we get
$$
\gamma_d\bigl(\rho(P + \eta {B_2^{d}})\bigr)
\le
\gamma_d\bigl(2\rho M({B_1^{\ell + d}})\bigr)
\le
\sum_{\substack{I\subset[\ell + d], |I| = d}}
\gamma_d\Bigl(2\rho \conv\{\pm Me_i: i\in I\}\Bigr).
$$
Each set $\conv\{\pm Me_i: i\in I\}$ is a (possibly degenerate) cross-polytope generated by vectors of norm at most
$$
\max\Bigl\{\max_{i\le \ell}\|x_i\|_2, \max_{j\le d}\|\sqrt d \eta e_j\|_2\Bigr\}
 = \max\{R,\sqrt d \eta\}\le R + \eta\sqrt d.
$$
Applying Lemma~\ref{lem:crosspoly-measure} (absorbing the factor $2$ into $C$) gives, uniformly over $I$,
$$
\gamma_d\Bigl(2\rho \conv\{\pm Me_i: i\in I\}\Bigr)\le
\Big(\frac{C\rho (R + \eta\sqrt d)}{d}\Big)^d.
$$
There are $\binom{\ell + d}{d}$ choices of $I$, so summing yields \eqref{eq:thickening-bound}.
\end{proof}

%%%%%%%%%%%%%%%%%%%%%%%%%%%%%%%%%%%%%%%%%%%
\subsection{Maurey sparsification and union inclusion} \label{sec:maurey}

Throughout this section, if $H$ is a Hilbert space with norm $\|\cdot\|_2$ induced by its inner product, we write
$$
B_2^{H} := \{h\in H: \|h\|_2\le 1\}.
$$

We use the following Maurey-type sparsification lemma (see, e.g.,
\cite[Expos\'e~5]{PisierMaurey1981}).

\begin{lemma}[Maurey lemma in a Hilbert space] \label{lem:maurey}
Let $(H,\langle\cdot,\cdot\rangle)$ be a Hilbert space and let $x_1,\dots,x_N\in H$ satisfy $\|x_i\|_2\le L$. Let $t\in\N$ and set $r_0 := L/\sqrt t$. Then for every $y\in \conv\{\pm x_1,\dots,\pm x_N\}$ there exists a subset $S\subset[N]$ with $|S|\le t$ such that
$$
\dist\bigl(y,\conv\{\pm x_i: i\in S\}\bigr)\le r_0.
$$
\end{lemma}

\begin{proof}
If $y=0$, pick any index $i_0\in[N]$ and set $S=\{i_0\}$. Since $0\in\conv\{\pm x_{i_0}\}$, we have $\dist(y,\conv\{\pm x_i:i\in S\})=0\le r_0$. Hence assume $y\neq 0$. Fix a representation $y = \sum_{i = 1}^N a_i x_i$ with $\sum_{i = 1}^N |a_i|\le 1$ (possible by definition of absolute convex hull). Let $\sigma_i := \sgn(a_i)$ (with $\sgn(0) := 0$) and define an $H$-valued random vector $\xi$ by
$$
\PP(\xi = \sigma_i x_i) = |a_i|\quad (i\in[N]),
\quad
\PP(\xi = 0) = 1-\sum_{i = 1}^N|a_i|.
$$
Then $\EE\xi = \sum_i a_i x_i = y$ and $\|\xi\|_2\le L$ almost surely.

Let $\xi_1,\dots,\xi_t$ be i.i.d. copies of $\xi$ and set
$$
Z := \frac{1}{t}\sum_{\ell = 1}^t \xi_\ell.
$$
Let $S\subset[N]$ be the set of indices $i$ such that $\xi_\ell\in\{\pm x_i\}$ for at least one $\ell$; then $|S|\le t$. Moreover, $Z\in \conv\{\pm x_i: i\in S\}$ because $Z$ is an average of $t$ signed atoms from $\{\pm x_i\}\cup\{0\}$.

Since $H$ is a Hilbert space,
$$
\EE\|Z-y\|_2^2
 = 
\EE\Big\|\frac{1}{t}\sum_{\ell = 1}^t(\xi_\ell-\EE\xi)\Big\|_2^2
 = 
\frac{1}{t}\EE\|\xi-\EE\xi\|_2^2
\le
\frac{1}{t}\EE\|\xi\|_2^2
\le
\frac{L^2}{t}.
$$
Thus $\EE\|Z-y\|_2\le \sqrt{\EE\|Z-y\|_2^2}\le L/\sqrt t = r_0$. Therefore there exists an outcome with $\|Z-y\|_2\le r_0$. For that outcome, $Z\in \conv\{\pm x_i: i\in S\}$ with $|S|\le t$, hence $\dist(y,\conv\{\pm x_i: i\in S\})\le r_0$.
\end{proof}

\begin{lemma}[Union inclusion after Maurey] \label{lem:maurey-union}
Let $H$ be a Hilbert space and let $\{x_i\}_{i = 1}^n\subset H$ satisfy $\|x_i\|_2\le L$.
Fix integers $1\le t\le k\le n$ and set $r_0 := L/\sqrt t$. For each $I\subset[n]$ define $Q_I := \conv\{\pm x_i: i\in I\}$. Then
$$
\bigcup_{\substack{I\subset[n], |I| = k}} Q_I
 \subset\
\bigcup_{\substack{S\subset[n], |S| = t}}\bigl(Q_S + r_0 B_2^{H}\bigr).
$$
\end{lemma}

\begin{proof}
Fix $I\subset[n]$ with $|I| = k$ and $y\in Q_I$. Apply Lemma~\ref{lem:maurey} to the family $\{x_i: i\in I\}$ to obtain $S\subset I$ with $|S|\le t$ and $\dist(y,Q_S)\le r_0$. If $|S|<t$, enlarge $S$ to some $S'\subset I$ with $|S'| = t$; by monotonicity $Q_S\subset Q_{S'}$, hence still $\dist(y,Q_{S'})\le r_0$. Thus $y\in Q_{S'} + r_0 B_2^{H}$ with $|S'| = t$, as claimed.
\end{proof}

%%%%%%%%%%%%%%%%%%%%%%%%%%%%%%%%%%%%%%%%%%%
\subsection{Union bound} \label{sec:largeU}

In this section we control
$$
p_k(A;\omega) = \gamma_n \bigl(4\rho \mathcal P^{(k)}(A;\omega)\bigr)
\quad (k \text{ large}),
$$
under the global event $\mathcal E^{(U)}$ from Definition~\ref{def:E3}. The argument consists of: 
\begin{enumerate}[leftmargin = 2.2em]
\item quotienting out the span of the chosen $K$-generators;
\item applying a Maurey-type sparsification inside the quotient space, with a new choice
$$
t = t(k) \asymp \frac{k}{\log(n\rho)};
$$
\item bounding Gaussian measure of the resulting Euclidean thickening using the thickening-stable estimate (Lemma~\ref{lem:thickening}).
\end{enumerate}

%%%%%%%%%%%%%%%%%%%%%%%%%%%%%%%%%%%%%%%%%%%
\subsection{Quotient bookkeeping}

For any Euclidean subspace $H\subset\R^n$ we write
$$
B_2^{H} := \{h\in H: \|h\|_2\le 1\},
\quad
\gamma_H := \mathcal L(\pi_H G)\quad\text{ for } G\sim N(0,\Id_n),
$$
where $\pi_H$ denotes the orthogonal projection onto $H$.

\begin{lemma}[Dimension reduction inside a subspace] \label{lem:dim-quotient}
Let $H\subset\R^n$ be a linear subspace with $\dim(H)\ge k$ and let $Q\subset H$ be measurable. Assume $\dim(\spann(Q))\le k$. Then there exists a $k$-dimensional subspace $\widetilde H\subset H$ with $\spann(Q)\subset \widetilde H$ such that for all $\eta\ge 0$,
$$
\gamma_H\bigl(Q + \eta B_2^{H}\bigr)
\le
\gamma_{\widetilde H}\bigl(Q + \eta B_2^{\widetilde H}\bigr).
$$
In particular, identifying $\widetilde H\simeq \R^k$ by an isometry $\Phi$, the right-hand side equals
$\gamma_k\bigl(\Phi(Q) + \eta B_2^k\bigr)$.
\end{lemma}

\begin{proof}
Let $E := \spann(Q)\subset H$, so $\dim(E)\le k$. Since $\dim(H)\ge k$, we may choose a $k$-dimensional subspace $\widetilde H\subset H$ with $E\subset \widetilde H$.

Let $\pi_{\widetilde H}:H\to \widetilde H$ be the orthogonal projection within the Hilbert space $H$. Because $Q\subset E\subset \widetilde H$, we have $\pi_{\widetilde H}(Q) = Q$. Also $\pi_{\widetilde H}(B_2^H) = B_2^{\widetilde H}$:
indeed, projection does not increase norm so $\pi_{\widetilde H}(B_2^H)\subset B_2^{\widetilde H}$, and conversely if $z\in B_2^{\widetilde H}$ then $z\in H$ with $\|z\|_2\le 1$ and $\pi_{\widetilde H}(z) = z$.

Hence for every $\eta\ge 0$,
$$
\pi_{\widetilde H}\bigl(Q + \eta B_2^H\bigr)
 = 
\pi_{\widetilde H}(Q) + \eta \pi_{\widetilde H}(B_2^H)
 = 
Q + \eta B_2^{\widetilde H}.
$$
Let $Z$ be a standard Gaussian vector in $H$ (i.e. with law $\gamma_H$). Then $\pi_{\widetilde H}Z$ is a standard Gaussian vector in $\widetilde H$ (i.e. with law $\gamma_{\widetilde H}$),
and
$$
\{Z\in Q + \eta B_2^H\}\subset \{\pi_{\widetilde H}Z\in \pi_{\widetilde H}(Q + \eta B_2^H)\}
 = 
\{\pi_{\widetilde H}Z\in Q + \eta B_2^{\widetilde H}\}.
$$
Taking probabilities yields
$$
\gamma_H(Q + \eta B_2^H)\le \gamma_{\widetilde H}(Q + \eta B_2^{\widetilde H}),
$$
as claimed. The final identification with $\gamma_k$ follows by applying an isometry $\Phi:\widetilde H\to\R^k$.
\end{proof}

%%%%%%%%%%%%%%%%%%%%%%%%%%%%%%%%%%%%%%%%%%%
\subsection{Union bound for $p_k$ in the large-$U$ regime}

Fix $k\in\{0,1,\dots,n\}$. Any $I\subset[2n]$ with $|I| = n$ and $|I\cap(n + [n])| = k$
can be written uniquely as
$$
I = I_1\cup(n + I_2),
\quad
I_1\subset[n], |I_1| = n-k,
\quad
I_2\subset[n], |I_2| = k.
$$
We write $\mathcal P_{I_1,I_2}(A;\omega) := \mathcal P_{I_1\cup(n + I_2)}(A;\omega)$.

\begin{proposition}[Large-$U$ regime: union measure bound] \label{prop:U-bound}
Assume $\mathcal E^{(U)}$ holds. Fix $A\in\mathcal A_\varepsilon$ and the induced exposed $\omega = \Gamma_{S(A)}$. Let $k\in\{1,\dots,n\}$ and set
\begin{equation} \label{eq:tdef}
\Lambda := \log(n\rho),\quad
t = t(k) := \Bigl\lceil\frac{k}{\Lambda}\Bigr\rceil,
\quad
r_0 := \frac{L}{\sqrt t},
\end{equation}
where $L$ is as in \eqref{eq:L-def}. Then
$$
p_k(A;\omega) := \gamma_n\Bigl(4\rho \mathcal P^{(k)}(A;\omega)\Bigr)
$$
satisfies
\begin{equation} \label{eq:pk-U-bound}
p_k(A;\omega)
\le
\binom{n}{k}\binom{n}{t}\binom{t + k}{k}
\Big(\frac{C\rho (L + \sqrt{k} r_0)}{k}\Big)^k,
\end{equation}
for a universal constant $C>0$.

Moreover, since $t\ge k/\Lambda$ one has $\sqrt{k} r_0\le L\sqrt{\Lambda}$; hence (for $n\ge 3$ and $\rho\ge 1$ so that $\Lambda\ge 1$)
$$
L + \sqrt{k} r_0 \le C L\sqrt{\log(n\rho)}.
$$
Consequently, using $\binom{n}{k}\le (en/k)^k$ and absorbing constants,
\begin{equation} \label{eq:pk-U-simplified}
p_k(A;\omega)
\le
\binom{n}{t}\binom{t + k}{k}
\Big(\frac{C\rho L n \sqrt{\log(n\rho)}}{k^2}\Big)^k.
\end{equation}
\end{proposition}

\begin{proof}
Fix $I_1\subset[n]$ with $|I_1| = n-k$ and define
$$
V := \spann\{v_i: i\in I_1\}\subset\R^n,
\quad
H := V^\perp.
$$
Then $\dim(H)\ge k$ because $\dim(V)\le n-k$.

For each $i\in[n]$ set $\tilde u_i := \pi_H(u_i)\in H$ and for $J\subset[n]$ define
$$
Q_J := \conv\{\pm \tilde u_i: i\in J\}\subset H.
$$
Since $\pi_H(v_i) = 0$ for all $i\in I_1$, Lemma~\ref{lem:proj-absconv} implies that for every $I_2\subset[n]$ with $|I_2| = k$,
$$
\pi_H\bigl(\mathcal P_{I_1,I_2}(A;\omega)\bigr)
 = 
\conv\Bigl(\{\pm\pi_H(v_i): i\in I_1\}\cup\{\pm \pi_H(u_i): i\in I_2\}\Bigr)
 = 
\conv\{\pm \tilde u_i: i\in I_2\}
 = 
Q_{I_2}.
$$
Therefore, by Lemma~\ref{lem:proj-monotone},
\begin{equation} \label{eq:quotient-step}
\gamma_n\Bigl(4\rho\bigcup_{|I_2| = k}\mathcal P_{I_1,I_2}(A;\omega)\Bigr)
\le
\gamma_H\Bigl(4\rho\bigcup_{|I_2| = k}Q_{I_2}\Bigr).
\end{equation}

On $\mathcal E^{(U)}$, \eqref{eq:U-gen-bound} gives $\|u_i\|_2\le L$ for all $i$, hence $\|\tilde u_i\|_2\le L$. Apply Lemma~\ref{lem:maurey-union} in the Hilbert space $H$ with parameters $(k,t)$ to obtain
$$
\bigcup_{|I_2| = k}Q_{I_2}
\subset
\bigcup_{|S| = t}\bigl(Q_S + r_0 B_2^{H}\bigr),
\quad r_0 = \frac{L}{\sqrt t}.
$$
Taking Gaussian measure and using the union bound gives
\begin{equation} \label{eq:union-S}
\gamma_H\Bigl(4\rho\bigcup_{|I_2| = k}Q_{I_2}\Bigr)
\le
\binom{n}{t}\max_{|S| = t}\gamma_H\bigl(4\rho(Q_S + r_0 B_2^{H})\bigr).
\end{equation}

Fix $S\subset[n]$ with $|S| = t$. Then $\dim(\spann(Q_S))\le t\le k$. Choose a $k$-dimensional subspace $\widetilde H\subset H$ containing $\spann(Q_S)$ and apply Lemma~\ref{lem:dim-quotient}:
$$
\gamma_H\bigl(4\rho(Q_S + r_0 B_2^{H})\bigr)
\le
\gamma_{\widetilde H}\bigl(4\rho(Q_S + r_0 B_2^{\widetilde H})\bigr).
$$
Identify $\widetilde H\simeq\R^k$ by an isometry. Apply Lemma~\ref{lem:thickening} in dimension $k$ with $\ell = t$ and thickening radius $\eta = r_0$. Since all generators of $Q_S$ have Euclidean norm at most $L$, we obtain
$$
\gamma_{\widetilde H}\bigl(4\rho(Q_S + r_0 B_2^{\widetilde H})\bigr)
\le
\binom{t + k}{k}\Big(\frac{C\rho (L + \sqrt{k} r_0)}{k}\Big)^k,
$$
absorbing the factor $4$ into the universal constant $C$.

Insert this into \eqref{eq:union-S} and then into \eqref{eq:quotient-step}. For the fixed $I_1$ we get
$$
\gamma_n\Bigl(4\rho\bigcup_{|I_2| = k}\mathcal P_{I_1,I_2}(A;\omega)\Bigr)
\le
\binom{n}{t}\binom{t + k}{k}\Big(\frac{C\rho (L + \sqrt{k} r_0)}{k}\Big)^k.
$$
Finally, take the union over all $I_1\subset[n]$ with $|I_1| = n-k$ (there are $\binom{n}{k}$ such choices) and apply a union bound:
$$
p_k(A;\omega)
 = 
\gamma_n\Bigl(4\rho \mathcal P^{(k)}(A;\omega)\Bigr)
\le
\binom{n}{k}\binom{n}{t}\binom{t + k}{k}\Big(\frac{C\rho (L + \sqrt{k} r_0)}{k}\Big)^k,
$$
which is \eqref{eq:pk-U-bound}.

For \eqref{eq:pk-U-simplified}, first note that $t = \lceil k/\Lambda\rceil$ implies $t\ge k/\Lambda$, hence $k/t\le \Lambda$. Therefore
$$
\sqrt{k} r_0 = \sqrt{k} \frac{L}{\sqrt t} = L\sqrt{\frac{k}{t}}\le L\sqrt{\Lambda}.
$$
Assuming $n\ge 3$ and $\rho\ge 1$ gives $\Lambda = \log(n\rho)\ge \log 3>1$, so $1 + \sqrt{\Lambda}\le 2\sqrt{\Lambda}$ and
$$
L + \sqrt{k}r_0 \le L(1 + \sqrt{\Lambda})\le 2L\sqrt{\Lambda}
 = 2L\sqrt{\log(n\rho)}.
$$
Also $\binom{n}{k}\le (en/k)^k$. Hence
$$
\binom{n}{k}\Big(\frac{C\rho (L + \sqrt{k} r_0)}{k}\Big)^k
\le
\Big(\frac{en}{k}\cdot \frac{C\rho\cdot 2L\sqrt{\log(n\rho)}}{k}\Big)^k
 = 
\Big(\frac{C\rho L n \sqrt{\log(n\rho)}}{k^2}\Big)^k,
$$
after adjusting $C$. Substituting into \eqref{eq:pk-U-bound} yields \eqref{eq:pk-U-simplified}.
\end{proof}

\begin{lemma}[Entropy bound for $\binom{n}{t(k)}$] \label{lem:entropy-tk}
Assume $n\ge 3$, $\rho\ge 1$, and set $\Lambda := \log(n\rho)\ge 1$. Let $k$ satisfy $\Lambda\le k\le n$ and define
$$
t(k) := \left\lceil\frac{k}{\Lambda}\right\rceil.
$$
Then
$$
\binom{n}{t(k)}\le \exp(C_1 k)
$$
for a universal constant $C_1>0$.
\end{lemma}

\begin{proof}
Let $t := t(k)$. Since $k\ge \Lambda$, we have $k/\Lambda\ge 1$, hence
$$
t = \left\lceil\frac{k}{\Lambda}\right\rceil \le \frac{k}{\Lambda} + 1 \le \frac{2k}{\Lambda}.
$$
Using the trivial bound $\binom{n}{t}\le (en/t)^t\le (en)^t$, we get
$$
\binom{n}{t}\le \exp\bigl(t\log(en)\bigr).
$$
Next, since $\rho\ge 1$, we have $\log n\le \log(n\rho) = \Lambda$, and therefore
$$
\log(en) = \log n + 1 \le \Lambda + 1 \le 2\Lambda,
$$
because $\Lambda\ge 1$. Hence
$$
t\log(en)\le \frac{2k}{\Lambda}\cdot 2\Lambda = 4k,
$$
and consequently $\binom{n}{t}\le e^{4k}$. The claim follows with $C_1 = 4$.
\end{proof}

%%%%%%%%%%%%%%%%%%%%%%%%%%%%%%%%%%%%%%%%%%%
%%%%%%%%%%%%%%%%%%%%%%%%%%%%%%%%%%%%%%%%%%%
\section{Completion of the proof and optimization} \label{sec:completion}

%%%%%%%%%%%%%%%%%%%%%%%%%%%%%%%%%%%%%%%%%%%
\subsection{Parameter choice and the $4/7$ exponent}

Fix $C\ge 1$ and $c_0\in(0,1]$. Define
$$
\rho = \rho(n) := c_0 n^{4/7}(\log n)^{-C}.
$$
Note that there exists $n_0 = n_0(C,c_0)$ such that for all $n\ge n_0$ one has $\rho\ge 1$ and
\begin{equation} \label{eq:log-comparability}
\log n \le \log(n\rho) \le 2\log n.
\end{equation}

\begin{lemma}[Verification of the parameter constraints]
 \label{lem:verify-constraints}
Fix the absolute constants appearing in the constraints:
\begin{itemize}
\item $c_{\mathrm{suppr}},c_{\mathrm{decay}}>0$ from Lemma~\ref{lem:suppression-measure};
\item $C_\star>0$ from Definition~\ref{def:EA1};
\item the absolute constant $C_{\mathrm{EK}}>0$ in Lemma~\ref{lem:EK-prob} (used in \eqref{eq:EK-conds});
\item the absolute constant $C_{\mathrm{U}}>0$ in Proposition~\ref{prop:U-bound} (i.e. the universal $C$ in \eqref{eq:pk-U-simplified}
and thus in \eqref{eq:rho-largeU-cond});
\item the constants $C_0>0$ (in $L$, see \eqref{eq:L-def}) and $C_1>0$ (in Lemma~\ref{lem:entropy-tk}).
\end{itemize}
Choose any $C\ge 2$ and any $c_0\in(0,1]$, and define $\rho$ as above. Let
$$
\Lambda := \log(n\rho).
$$
Let
\begin{equation} \label{eq:params-choice-47}
\tilde s := \Bigl\lceil n^{6/7}\Lambda^{2/7}\Bigr\rceil,\quad
r := \Bigl\lceil C_2 \tilde s \Lambda\Bigr\rceil,\quad
s := \Bigl\lfloor \frac{\tilde s^2}{C_3 n \Lambda}\Bigr\rfloor,
\end{equation}
where $C_2,C_3>0$ are absolute constants chosen as follows:
$$
C_3\ge C_{\mathrm{EK}},\quad
C_2\ge \max\Bigl\{4,\frac{16}{c_{\mathrm{decay}}}\Bigr\}.
$$
Then there exist $n_1$ and $c_0^\ast\in(0,1]$ (depending only on the displayed absolute constants)
such that if $c_0\le c_0^\ast$ and $n\ge n_1$, all parameter constraints
\eqref{eq:EK-conds}, \eqref{eq:r-beats-comb}, \eqref{eq:scale}, and \eqref{eq:rho-largeU-cond} hold,
and in addition
\begin{equation} \label{eq:tilde-s-ge-Lambda}
\tilde s \ge \Lambda.
\end{equation}
\end{lemma}

\begin{proof}
See Appendix~\ref{sec:verify-constraints}
\end{proof}

Set $\Lambda := \log(n\rho)$.
Under the standard choice
$$
r \asymp \tilde s \Lambda
\quad\text{ and}\quad
s \asymp \frac{\tilde s^2}{n\Lambda},
$$
the two regimes impose the following dominant (polynomial) constraints:

\subsubsection*{(A) Small-$U$ (scaling).}

The scaling requirement \eqref{eq:scale} becomes, up to absolute constants,
\begin{equation} \label{eq:constraint-smallU-47}
\rho \lesssim \frac{n}{\sqrt r}
 \asymp \frac{n}{\sqrt{\tilde s \Lambda}}.
\end{equation}

\subsubsection*{(B) Large-$U$ (Maurey and thickening tail).}

The large-$U$ tail condition \eqref{eq:rho-largeU-cond}, together with the choice
$s\asymp \tilde s^2/(n\Lambda)$, implies
$$
L \asymp \sqrt{\frac{n}{s}\Lambda}\asymp \frac{n\Lambda}{\tilde s},
$$
and therefore \eqref{eq:rho-largeU-cond} is (again up to absolute constants) equivalent to
\begin{equation} \label{eq:constraint-largeU-47}
\rho \lesssim \frac{\tilde s^3}{n^2 \Lambda^{3/2}}.
\end{equation}

\subsubsection*{Balancing.}

Equating the right-hand sides of \eqref{eq:constraint-smallU-47} and \eqref{eq:constraint-largeU-47} yields
$$
\frac{n}{\sqrt{\tilde s \Lambda}}
 \asymp\
\frac{\tilde s^3}{n^2 \Lambda^{3/2}}
\quad\Longleftrightarrow\quad
\tilde s^{7/2} \asymp n^3 \Lambda.
$$
Hence the balanced choice is
\begin{equation} \label{eq:balanced-tilde-s-47}
\tilde s \asymp (n^3\Lambda)^{2/7}
 = n^{6/7}\Lambda^{2/7}.
\end{equation}
Substituting \eqref{eq:balanced-tilde-s-47} into \eqref{eq:constraint-smallU-47} gives
$$
\rho \asymp \frac{n}{\sqrt{\tilde s \Lambda}}
 = 
\frac{n}{\sqrt{n^{6/7}\Lambda^{2/7}\cdot \Lambda}}
 = 
n^{4/7}\Lambda^{-9/14}.
$$
Finally, by~\eqref{eq:log-comparability} one has $\Lambda = \log(n\rho)\asymp\log n$ for polynomial $\rho$,
so the balanced scales can be written as
$$
\tilde s \asymp n^{6/7}(\log n)^{2/7},
\quad
\rho \asymp n^{4/7}(\log n)^{-9/14},
$$
up to universal constants and additional (harmless) logarithmic slack.

\begin{remark}[Logarithmic bookkeeping] \label{rem:log-slack}
The heuristic optimization in this section suggests the more specific scale $\rho \asymp n^{4/7}\Lambda^{-9/14}$ with $\Lambda = \log(n\rho)$. In Theorem~\ref{thm:main} we allow an arbitrary $(\log n)^{-C}$ loss to avoid carrying explicit log-exponents through several nested union bounds (most notably in the discretization net and the uniform tilt event) and through auxiliary constraints such as $\tilde s\ge \log^2 n$. A more systematic tracking of these logarithmic contributions should yield an explicit power of $\log n$ in the final statement, at the expense of additional bookkeeping.
\end{remark}

%%%%%%%%%%%%%%%%%%%%%%%%%%%%%%%%%%%%%%%%%%%
\subsection{Proof of Theorem~\ref{thm:main}} \label{subsec:proof-main}

We also use the following large-$U$ scale condition (compare Lemma~\ref{lem:verify-constraints}):
\begin{equation} \label{eq:rho-largeU-cond}
\frac{C_{\mathrm U} \rho L n \sqrt{\Lambda}}{\tilde s^2} \le \frac{1}{4e^{2C_1 + 1}},
\end{equation}
where $\Lambda := \log(n\rho)$, $L$ is as in \eqref{eq:L-def}, $C_1$ is the constant from Lemma~\ref{lem:entropy-tk}, and $C_{\mathrm U}>0$ is the absolute constant implicit in \eqref{eq:pk-U-simplified}.

\begin{proof}[Proof of Theorem~\ref{thm:main}]
Fix the absolute constants from Lemmas~\ref{lem:suppression-measure} and~\ref{lem:crosspoly-measure}, Lemma~\ref{lem:EK-prob}, Proposition~\ref{prop:U-bound}, and Lemma~\ref{lem:entropy-tk}, and choose $C\ge 2$ and $c_0\in(0,1]$.
Set
$$
m := n^3,\quad
\rho := c_0 n^{4/7}(\log n)^{-C},\quad
\Lambda := \log(n\rho),\quad
\varepsilon := \frac{1}{\rho n^2}.
$$
By~\eqref{eq:log-comparability}, for $n\ge n_0(C,c_0)$ one has $\rho\ge 1$ and $\log n\le \Lambda\le 2\log n$.

Choose parameters $(\tilde s,r,s)$ as in \eqref{eq:params-choice-47}. By Lemma~\ref{lem:verify-constraints}, if $c_0\le c_0^\ast$ and $n\ge n_1$, then the constraints \eqref{eq:EK-conds}, \eqref{eq:r-beats-comb}, \eqref{eq:scale}, and \eqref{eq:rho-largeU-cond} hold, and in addition $\tilde s\ge \Lambda$.

Let $\mathcal A_\varepsilon = \mathcal A_{m,n}(\varepsilon)$ be the discretization net from Lemma~\ref{lem:discretization}. Recall that $\varepsilon\rho n^2 = 1$, so the hypothesis \eqref{eq:disc-epscond} is satisfied. By Lemma~\ref{lem:discretization} and \eqref{eq:disc-prob},
\begin{equation} \label{eq:Erho-net}
\PP(\mathcal E_\rho)
\le
\PP\Bigl(\exists A\in\mathcal A_\varepsilon:\quad G_m\subset 2\rho \Gamma A(B_1^n)\Bigr).
\end{equation}

Define the ``good'' event
$$
\mathcal G := \mathcal E^{(U)} \cap \mathcal E^{(K)}(r,\tilde s).
$$
By Lemma~\ref{lem:E3prob} and Lemma~\ref{lem:EK-prob},
\begin{equation} \label{eq:G-prob}
\PP(\mathcal G^c)\le \exp\bigl(-c n\log(n\rho)\bigr) + \frac{1}{n}.
\end{equation}

\subsubsection*{Step 1: uniform bound on $\gamma_n(\mathcal K_A)$ on $\mathcal G$.}

Fix $A\in\mathcal A_\varepsilon$ and let $\omega = \Gamma_{S(A)}$ be the exposed realization. Recall $\mathcal K_A(\rho;\omega)$ from \eqref{eq:KAdef}. By Lemma~\ref{lem:bridge},
$$
\gamma_n\bigl(\mathcal K_A(\rho;\omega)\bigr)
\le p(A;\omega)
\le \sum_{k = 0}^n p_k(A;\omega),
\quad
p_k(A;\omega) := \gamma_n\bigl(4\rho \mathcal P^{(k)}(A;\omega)\bigr).
$$

On the event $\mathcal E^{(K)}(r,\tilde s)$, Proposition~\ref{prop:K-bound} and the constraints \eqref{eq:scale}--\eqref{eq:r-beats-comb} yield
\begin{equation} \label{eq:smallU-part}
\sum_{k = 0}^{\tilde s} p_k(A;\omega)\le \frac14.
\end{equation}

On the event $\mathcal E^{(U)}$, Proposition~\ref{prop:U-bound} gives for every $k\in\{1,\dots,n\}$ the bound \eqref{eq:pk-U-simplified}.
Fix $k\in\{\tilde s + 1,\dots,n\}$. Since $\tilde s\ge \Lambda$, we have $k\ge \Lambda$ and Lemma~\ref{lem:entropy-tk} applies, yielding
$$
\binom{n}{t(k)} \le \exp(C_1 k),
\quad
t(k) = \Bigl\lceil \frac{k}{\Lambda}\Bigr\rceil \le k.
$$

Also, since $t(k)=\lceil k/\Lambda\rceil\le k$, we have
$$
\binom{t(k)+k}{k}\le 2^{t(k)+k}\le 2^{2k}=4^k=e^{k\log 4}\le e^{2k}.
$$
Therefore \eqref{eq:pk-U-simplified} and Lemma~\ref{lem:entropy-tk} give
$$
p_k(A;\omega)
\le \binom{n}{t(k)}\binom{t(k)+k}{k}
\Bigl(\frac{C_{\mathrm U}\rho L n\sqrt{\Lambda}}{k^2}\Bigr)^k
\le e^{C_1k}\cdot e^{2k}\Bigl(\frac{C_{\mathrm U}\rho L n\sqrt{\Lambda}}{k^2}\Bigr)^k
= e^{(C_1+2)k}\Bigl(\frac{C_{\mathrm U}\rho L n\sqrt{\Lambda}}{k^2}\Bigr)^k.
$$
Since $k\ge \tilde s+1$, we have $k^2\ge \tilde s^2$, hence
$$
p_k(A;\omega)
\le e^{(C_1+2)k}\Bigl(\frac{C_{\mathrm U}\rho L n\sqrt{\Lambda}}{\tilde s^2}\Bigr)^k.
$$
Using \eqref{eq:rho-largeU-cond}, we get
$$
\frac{C_{\mathrm U}\rho L n\sqrt{\Lambda}}{\tilde s^2}\le \frac{1}{4e^{2C_1+1}},
$$
and therefore
$$
p_k(A;\omega)\le \Bigl(\frac{e^{C_1+2}}{4e^{2C_1+1}}\Bigr)^k
=\Bigl(\frac{e^{1-C_1}}{4}\Bigr)^k
\le \Bigl(\frac14\Bigr)^k
\quad\text{since } C_1\ge 1.
$$

Using \eqref{eq:rho-largeU-cond} and $C_1\ge 1$, we obtain $p_k(A;\omega)\le (1/4)^k$ and therefore
\begin{equation} \label{eq:largeU-part}
\sum_{k = \tilde s + 1}^n p_k(A;\omega)
\le
\sum_{k = \tilde s + 1}^\infty \Big(\frac14\Big)^k
\le
\frac14.
\end{equation}

Combining \eqref{eq:smallU-part} and \eqref{eq:largeU-part} yields
$$
\sum_{k = 0}^n p_k(A;\omega)
 = 
\sum_{k = 0}^{\tilde s} p_k(A;\omega) + \sum_{k = \tilde s + 1}^{n} p_k(A;\omega)
\le \frac12,
$$
and hence
$$
\gamma_n\bigl(\mathcal K_A(\rho;\omega)\bigr)\le \frac12.
$$
uniformly for all $A\in\mathcal A_\varepsilon$ on the event $\mathcal G$.

\subsubsection*{Step 2: powering beats the discretization net.}

For each $A\in\mathcal A_\varepsilon$ define the (exposed) event
$$
\mathcal G_A := \Bigl\{\gamma_n\bigl(\mathcal K_A(\rho;\Gamma_{S(A)})\bigr)\le \frac12\Bigr\}.
$$
Since $\mathcal K_A(\rho;\Gamma_{S(A)})$ depends only on $\Gamma_{S(A)}$, the event $\mathcal G_A$ is $\mathcal F_A$-measurable.
Moreover, Step~1 shows that on $\mathcal G$ one has $\mathcal G\subset \bigcap_{A\in\mathcal A_\varepsilon}\mathcal G_A$.

Fix $A\in\mathcal A_\varepsilon$.
By Lemma~\ref{lem:powering} and \eqref{eq:powering-cond},
$$
\PP\Bigl(G_m\subset 2\rho \Gamma A(B_1^n) \Big| \mathcal F_A\Bigr)
\le
\gamma_n\bigl(\mathcal K_A(\rho;\Gamma_{S(A)})\bigr)^{N(A)}.
$$
Therefore,
$$
\PP\Bigl(\bigl\{G_m\subset 2\rho \Gamma A(B_1^n)\bigr\}\cap \mathcal G_A\Bigr)
 = 
\EE\Bigl[\mathbf 1_{\mathcal G_A} 
\PP\bigl(G_m\subset 2\rho \Gamma A(B_1^n) \mid \mathcal F_A\bigr)\Bigr]
\le
\EE\Bigl[\mathbf 1_{\mathcal G_A} \Big(\frac12\Big)^{N(A)}\Bigr]
\le
2^{-N(A)}.
$$
Since $|S(A)|\le n^2$ by \eqref{eq:SA-size} and $m = n^3$, one has $N(A) = m-|S(A)|\ge n^3-n^2$, hence
$$
\PP\Bigl(\bigl\{G_m\subset 2\rho \Gamma A(B_1^n)\bigr\}\cap \mathcal G_A\Bigr)
\le
2^{-(n^3-n^2)}.
$$
Union bounding over $A\in\mathcal A_\varepsilon$ gives
\begin{equation} \label{eq:net-powering}
\PP\Bigl(\exists A\in\mathcal A_\varepsilon: G_m\subset 2\rho \Gamma A(B_1^n) \text{ and } \mathcal G\Bigr)
\le
\PP\Bigl(\exists A\in\mathcal A_\varepsilon: G_m\subset 2\rho \Gamma A(B_1^n) \text{ and } \mathcal G_A\Bigr)
\le
|\mathcal A_\varepsilon| 2^{-(n^3-n^2)}.
\end{equation}
By \eqref{eq:cardAeps-log}, $\log|\mathcal A_\varepsilon|\le C n^2\log(n\rho) = o(n^3)$, hence the right-hand side of \eqref{eq:net-powering}
is at most $\exp(-c'n^3)$ for all sufficiently large $n$.

\subsubsection*{Step 3: conclude.}

Combining \eqref{eq:Erho-net}, \eqref{eq:net-powering}, and \eqref{eq:G-prob} yields, for all sufficiently large $n$,
$$
\PP(\mathcal E_\rho)
\le
\PP(\mathcal G^c)
 + 
\PP\Bigl(\exists A\in\mathcal A_\varepsilon: G_m\subset 2\rho \Gamma A(B_1^n) \text{ and } \mathcal G\Bigr)
\le
\frac{1}{n} + \exp\bigl(-c n\log(n\rho)\bigr) + \exp(-c'n^3)
\le
\frac{2}{n}.
$$
Since $\mathcal E_\rho = \{d_{\mathrm{BM}}(G_m, B_1^n)\le \rho\}$ by \eqref{eq:E-rho}, we obtain
$$
\PP\Bigl\{d_{\mathrm{BM}}(G_m, B_1^n)\ge \rho\Bigr\}\ge 1-\frac{2}{n}.
$$
Recalling $\rho = c_0 n^{4/7}(\log n)^{-C}$ proves the theorem with $c := c_0$ (after possibly decreasing $c_0$ and increasing $C$ by absolute factors).
\end{proof}

%%%%%%%%%%%%%%%%%%%%%%%%%%%%%%%%%%%%%%%%%%%
\appendix

%%%%%%%%%%%%%%%%%%%%%%%%%%%%%%%%%%%%%%%%%%%
%%%%%%%%%%%%%%%%%%%%%%%%%%%%%%%%%%%%%%%%%%%
\section{Standard auxiliary results} \label{app:standard}

%%%%%%%%%%%%%%%%%%%%%%%%%%%%%%%%%%%%%%%%%%%
\subsection{Elementary convexity and combinatorics}

\begin{lemma}[Cross-polytopes as linear images and volume of ${B_1^{d}}$] \label{lem:crosspoly-linear-volume}
Let $d\in\N$ and let $y_1,\dots,y_d\in\R^d$. Let $Y = [y_1 \cdots y_d]\in\R^{d\times d}$.
Then
$$
\conv\{\pm y_1,\dots,\pm y_d\} = Y({B_1^{d}}).
$$
Moreover,
$$
\mathrm{vol}_d({B_1^{d}}) = \frac{2^d}{d!}.
$$
\end{lemma}

\begin{proof}
For the first claim, recall that
$$
\conv\{\pm y_1,\dots,\pm y_d\}
 = 
\Bigl\{\sum_{i = 1}^d a_i y_i: \sum_{i = 1}^d |a_i|\le 1\Bigr\}.
$$
Since $Y a = \sum_{i = 1}^d a_i y_i$, the right-hand side equals $\{Ya: \|a\|_1\le 1\} = Y({B_1^{d}})$.

For the volume formula, note that ${B_1^{d}} = \conv\{\pm e_1,\dots,\pm e_d\}$, where $(e_i)$ is the standard basis.
For each sign vector $\sigma\in\{\pm 1\}^d$ define the simplex
$$
\Delta_\sigma := \conv\{0,\sigma_1 e_1,\dots,\sigma_d e_d\}.
$$
Then ${B_1^{d}} = \bigcup_{\sigma\in\{\pm 1\}^d}\Delta_\sigma$ and the interiors of these simplices are disjoint.
Each $\Delta_\sigma$ has volume $1/d!$ (its edge vectors form a unimodular matrix), hence
$$
\mathrm{vol}_d({B_1^{d}}) = 2^d\cdot \frac{1}{d!}.
$$
\end{proof}

\begin{lemma}[Absorption of a small Minkowski term] \label{lem:absorb}
Let $K,L\subset\R^n$ be origin-symmetric convex bodies and let $\delta\in[0,1)$.
If
$$
K\subset L + \delta K,
$$
then
$$
K\subset \frac{1}{1-\delta} L,
\quad\text{ equivalently}\quad
(1-\delta)K\subset L.
$$
\end{lemma}

\begin{proof}
For a nonempty compact convex set $S\subset\R^n$ define its support function
$$
h_S(u) := \sup_{x\in S}\langle u,x\rangle,\quad u\in\R^n.
$$
Then $h_{S + T} = h_S + h_T$ and $h_{\lambda S} = \lambda h_S$ for all $\lambda\ge0$. The inclusion $K\subset L + \delta K$ implies
$$
h_K\le h_{L + \delta K} = h_L + \delta h_K,
$$
hence $(1-\delta)h_K\le h_L$, i.e. $h_K\le \frac{1}{1-\delta}h_L = h_{\frac{1}{1-\delta}L}$.
Finally, for compact convex $S,T$, one has $S\subset T$ iff $h_S\le h_T$ (by Hahn--Banach separation).
Therefore $K\subset \frac{1}{1-\delta}L$.
\end{proof}

\begin{lemma} \label{lem:binombound}
For integers $1\le r\le n$,
$$
\binom{n}{r}\le \Big(\frac{en}{r}\Big)^r.
$$
\end{lemma}

\begin{proof}
Using $r!\ge (r/e)^r$,
$$
\binom{n}{r} = \frac{n(n-1)\cdots(n-r + 1)}{r!}\le \frac{n^r}{(r/e)^r} = \Big(\frac{en}{r}\Big)^r.
$$
\end{proof}

%%%%%%%%%%%%%%%%%%%%%%%%%%%%%%%%%%%%%%%%%%%
\subsection{Gaussian concentration and tails}

\begin{lemma}[Gaussian concentration for Lipschitz maps] \label{lem:gauss-concentration}
Let $f:\R^n\to\R$ be $1$-Lipschitz with respect to $\|\cdot\|_2$, and let $G\sim N(0,\Id_n)$.
Then for every $t\ge 0$,
$$
\PP\{f(G)\ge \EE f(G) + t\}\le \exp(-t^2/2).
$$
\end{lemma}

\begin{proof}
This is a standard consequence of the Gaussian isoperimetric inequality (Borell's inequality).
See, e.g., \cite[Chapter~2]{Ledoux2001} or \cite{LedouxTalagrand1991}.
\end{proof}

\begin{lemma}[Gaussian norm tail] \label{lem:gauss-norm-tail}
There exist universal constants $c,C>0$ such that for all $n\in\mathbb N$, all $t\ge 0$, and $G\sim N(0,\Id_n)$,
$$
\PP\{\|G\|_2\ge \sqrt{n} + t\} \le \exp(-c t^2),
\quad
\PP\{\|G\|_2\ge C\sqrt{n + t}\} \le \exp(-c t).
$$
\end{lemma}

\begin{proof} ~

\subsubsection*{First inequality.}

The function $f:\mathbb R^n\to\mathbb R$, $f(x) = \|x\|_2$, is $1$-Lipschitz, hence by Lemma~\ref{lem:gauss-concentration},
for all $t\ge 0$,
$$
\PP\{f(G)\ge \EE f(G) + t\}\le \exp(-t^2/2).
$$
Moreover, by Cauchy--Schwarz, $\EE\|G\|_2\le \sqrt{\EE\|G\|_2^2} = \sqrt{n}$. Therefore,
$$
\PP\{\|G\|_2\ge \sqrt{n} + t\}
\le \PP\{\|G\|_2\ge \EE\|G\|_2 + t\}
\le \exp(-t^2/2),
$$
which gives the first inequality.

\subsubsection*{Second inequality.}

Let $X := \|G\|_2^2$ so that $X\sim \chi^2_n$. We use the standard upper tail estimate:

\begin{claim}[Chi-square upper tail] \label{clm:chi-square-tail}
For all $t\ge 0$,
$$
\PP\{X-n\ge 2\sqrt{nt} + 2t\}\le \exp(-t).
$$
\end{claim}

\begin{proof}[Proof of Claim~\ref{clm:chi-square-tail}]
Write $X = \sum_{i = 1}^n Z_i^2$ with $Z_i\sim N(0,1)$ i.i.d. For $\lambda\in(0,1/2)$,
$$
\EE e^{\lambda Z_1^2} = (1-2\lambda)^{-1/2},
\quad\Rightarrow\quad
\EE e^{\lambda X} = (1-2\lambda)^{-n/2}.
$$
Hence
$$
\log \EE e^{\lambda (X-n)}
 = -\frac{n}{2}\log(1-2\lambda) - \lambda n.
$$
Using $-\log(1-u)\le u + \frac{u^2}{2(1-u)}$ for $u\in(0,1)$ and substituting $u = 2\lambda$ gives
$$
\log \EE e^{\lambda (X-n)}
\le \frac{n\lambda^2}{1-2\lambda},
\quad 0<\lambda<\frac12.
$$
By Chernoff's bound, for any $y\ge 0$,
$$
\PP\{X-n\ge y\}
\le \exp\Bigl(-\lambda y + \frac{n\lambda^2}{1-2\lambda}\Bigr),\quad 0<\lambda<\frac12.
$$
Fix $t\ge 0$ and set $y := 2\sqrt{nt} + 2t$. Let $\alpha := \sqrt{t/n}$ and choose $\lambda := \frac{\alpha}{1 + 2\alpha}\in(0,1/2)$.
Then $1-2\lambda = \frac{1}{1 + 2\alpha}$, $\frac{n\lambda^2}{1-2\lambda} = \frac{t}{1 + 2\alpha}$, and
$$
\lambda y
 = \frac{\alpha}{1 + 2\alpha}\cdot 2n\alpha(1 + \alpha)
 = \frac{2t(1 + \alpha)}{1 + 2\alpha}.
$$
Thus $-\lambda y + \frac{n\lambda^2}{1-2\lambda} = -t$, proving the claim.
\end{proof}

Note that for all $t\ge 0$,
$$
(\sqrt{n} + \sqrt{2t})^2 - n
 = 2t + 2\sqrt{2nt}
\ge 2t + 2\sqrt{nt}.
$$
Hence $\{\|G\|_2\ge \sqrt{n} + \sqrt{2t}\}\subset \{X-n\ge 2\sqrt{nt} + 2t\}$ and by the claim,
$$
\PP\{\|G\|_2\ge \sqrt{n} + \sqrt{2t}\}\le e^{-t}.
$$
Finally, $\sqrt{n} + \sqrt{2t}\le 2\sqrt{n + t}$ since $(\sqrt{n} + \sqrt{2t})^2\le 4(n + t)$, so
$$
\PP\{\|G\|_2\ge 2\sqrt{n + t}\}\le e^{-t}.
$$
Absorbing constants yields the stated second inequality.
\end{proof}

%%%%%%%%%%%%%%%%%%%%%%%%%%%%%%%%%%%%%%%%%%%
\subsection{$\ell_1$-sparsity via linear programming}

The next lemma is a standard linear-programming fact: an $\ell_1$-minimizer can be chosen sparse, with linearly independent support columns.
We include a complete proof for completeness.

\begin{lemma} \label{lem:l1-sparse-existence}
Let $M\in\R^{p\times N}$ and set $d := \rank(M)$. Fix $y\in M(B_1^N)$ and define
$$
\tau := \min\{\|x\|_1: Mx = y\}.
$$
Then $\tau\le 1$, and there exists a minimizer $x^\star\in\R^N$ such that
$$
Mx^\star = y,\quad \|x^\star\|_1 = \tau,\quad |\supp(x^\star)|\le d,
$$
and moreover the columns $\{Me_i: i\in\supp(x^\star)\}$ are linearly independent, i.e.
$$
\rank(M_{\supp(x^\star)}) = |\supp(x^\star)|.
$$
\end{lemma}

\begin{proof}
Since $y\in M(B_1^N)$, there exists $x^{(0)}$ with $Mx^{(0)} = y$ and $\|x^{(0)}\|_1\le 1$, hence $\tau\le 1$.

Consider the polyhedron
$$
\mathcal P := \{z\in\R^{2N}: Az = y, z\ge 0\},
\quad
A := [M -M]\in\R^{p\times 2N}.
$$
Write $x = u-v$ with $u,v\ge 0$ and $z = (u,v)$. Then $Az = y$ and $\mathbf 1^\top z = \|u\|_1 + \|v\|_1$.
The feasible set $\mathcal P$ is nonempty because $x^{(0)}$ is feasible.
Moreover, the objective $\mathbf 1^\top z$ is minimized on $\mathcal P$:
indeed, the sublevel set $\{z\in\mathcal P:\mathbf 1^\top z\le \mathbf 1^\top z^{(0)}\}$ is closed,
and bounded since $z\ge 0$ and $\sum_i z_i\le \mathbf 1^\top z^{(0)}$ implies $0\le z_i\le \mathbf 1^\top z^{(0)}$ for all $i$.
Hence this sublevel set is compact and the minimum is attained.

Let $\mathcal F := \{z\in\mathcal P: \mathbf 1^\top z = \tau\}$ be the nonempty face of minimizers and fix an extreme point $z^\star\in\mathcal F$.
Write $z^\star = (u^\star,v^\star)$ and set $x^\star := u^\star-v^\star$. Then $Mx^\star = y$ and $\mathbf 1^\top z^\star = \tau$.

\subsubsection*{No cancellations.}

If $u^\star_i>0$ and $v^\star_i>0$ for some $i$, let $t := \min\{u^\star_i,v^\star_i\}>0$ and define
$u' := u^\star-te_i$, $v' := v^\star-te_i$, $z' := (u',v')$.
Then $z'\ge 0$, $Az' = Az^\star = y$, and $\mathbf 1^\top z' = \mathbf 1^\top z^\star-2t<\tau$, contradicting minimality.
Hence $u^\star_i v^\star_i = 0$ for all $i$, and therefore
$$
\|x^\star\|_1 = \|u^\star\|_1 + \|v^\star\|_1 = \mathbf 1^\top z^\star = \tau.
$$

\subsubsection*{Support size and independence.}

Let $J := \supp(z^\star) = \{k: z^\star_k>0\}\subset[2N]$.
We claim that the columns of $A$ indexed by $J$ are linearly independent.
Indeed, if they were dependent then there exists $0\neq h\in\R^{2N}$ supported on $J$ with $Ah = 0$.
Choose
$$
0<\eta<\min_{k\in J, h_k\neq 0}\frac{z^\star_k}{|h_k|}.
$$
Then $z^\star\pm \eta h\ge 0$ and $A(z^\star\pm \eta h) = Az^\star = y$, so $z^\star\pm \eta h\in\mathcal P$ are distinct and
$$
z^\star = \frac12(z^\star + \eta h) + \frac12(z^\star-\eta h),
$$
contradicting extremality. Thus $A_J$ has full column rank, and hence
$$
|J|\le \rank(A) = \rank(M) = d.
$$

By the no-cancellation property, at most one of $i$ and $N + i$ belongs to $J$ for each $i\in[N]$.
Let $I := \supp(x^\star)\subset[N]$. Then $|I| = |J|$ and the multiset of columns $\{Ae_k: k\in J\}$ equals $\{\pm Me_i: i\in I\}$.
Multiplying columns by $-1$ does not affect linear independence, so $\rank(M_I) = |I|$.
Finally $|I| = |J|\le d$.
\end{proof}

\begin{corollary} \label{cor:l1-decomp}
Let $M\in\R^{p\times N}$ have rank $d := \rank(M)$. Then
$$
M(B_1^N)
\subset
\bigcup_{\substack{I\subset[N], |I|\le d, \rank(M_I) = |I|}}
\conv\{\pm Me_i: i\in I\}.
$$
\end{corollary}

\begin{proof}
Fix $y\in M(B_1^N)$. By Lemma~\ref{lem:l1-sparse-existence}, there exists $x^\star$ with $Mx^\star = y$, $\|x^\star\|_1\le 1$,
and $I := \supp(x^\star)$ satisfies $|I|\le d$ and $\rank(M_I) = |I|$.
Write $y = \sum_{i\in I} x^\star_i Me_i$ with $\sum_{i\in I}|x^\star_i| = \|x^\star\|_1\le 1$; hence
$y\in \conv\{\pm Me_i: i\in I\}$.
\end{proof}

%%%%%%%%%%%%%%%%%%%%%%%%%%%%%%%%%%%%%%%%%%%
\subsection{Elementary Gaussian-volume bounds}

\begin{lemma}[Gaussian measure via volume] \label{lem:gauss-volume}
For any measurable $S\subset\R^d$,
$$
\gamma_d(S)\le (2\pi)^{-d/2} \mathrm{vol}_d(S).
$$
\end{lemma}

\begin{proof}
The Gaussian density on $\R^d$ satisfies $(2\pi)^{-d/2}e^{-\|x\|_2^2/2}\le (2\pi)^{-d/2}$ for all $x$.
Integrating over $S$ yields the claim.
\end{proof}

\begin{lemma}[Measure of a cross-polytope] \label{lem:crosspoly-measure}
Let $y_1,\dots,y_d\in\R^d$ and set $R := \max_{i\le d}\|y_i\|_2$. Then for all $\rho\ge 1$,
$$
\gamma_d\Big(\rho \conv\{\pm y_1,\dots,\pm y_d\}\Big)
\le \Big(\frac{C\rho R}{d}\Big)^d.
$$
\end{lemma}

\begin{proof}
Let $Y = [y_1 \cdots y_d]\in\R^{d\times d}$. Then $\conv\{\pm y_i\} = Y({B_1^{d}})$ and hence
$$
\mathrm{vol}_d\big(\rho Y({B_1^{d}})\big) = \rho^d|\det Y| \mathrm{vol}_d({B_1^{d}}).
$$
By Hadamard's inequality, $|\det Y|\le \prod_{i = 1}^d\|y_i\|_2\le R^d$.
Also $\mathrm{vol}_d({B_1^{d}}) = 2^d/d!\le (2e/d)^d$.
Applying Lemma~\ref{lem:gauss-volume} gives
$$
\gamma_d\big(\rho Y({B_1^{d}})\big)\le (2\pi)^{-d/2}\rho^d R^d \frac{2^d}{d!}\le \Big(\frac{C\rho R}{d}\Big)^d,
$$
after absorbing constants. (If $\det Y = 0$ then the left-hand side is $0$ and the inequality is trivial.)
\end{proof}

%%%%%%%%%%%%%%%%%%%%%%%%%%%%%%%%%%%%%%%%%%%
\subsection{Projection facts}

\begin{lemma}[Projection monotonicity] \label{lem:proj-monotone}
Let $H\subset\R^n$ be a linear subspace and $\pi_H$ the orthogonal projection.
Then for every measurable set $S\subset\R^n$,
$$
\gamma_n(S)\le \gamma_H(\pi_H S).
$$
\end{lemma}

\begin{proof}
Let $G\sim N(0,\Id_n)$. Then $\pi_H G$ has law $\gamma_H$ by definition.
Moreover, $\{G\in S\}\subset\{\pi_H G\in \pi_H S\}$.
Taking probabilities gives
$$
\gamma_n(S) = \PP\{G\in S\}\le \PP\{\pi_H G\in \pi_H S\} = \gamma_H(\pi_H S).
$$
\end{proof}

\begin{lemma}[Projection of absolute convex hull] \label{lem:proj-absconv}
Let $H$ be a Hilbert space, $\pi:H\to H$ a linear map, and $a_1,\dots,a_N\in H$.
Then
$$
\pi \left(\conv\{\pm a_1,\dots,\pm a_N\}\right)
 = 
\conv\{\pm \pi(a_1),\dots,\pm \pi(a_N)\}.
$$
\end{lemma}

\begin{proof}
Recall the standard representation
$$
\conv\{\pm a_1,\dots,\pm a_N\}
 = 
\Bigl\{\sum_{i = 1}^N \alpha_i a_i: \sum_{i = 1}^N |\alpha_i|\le 1\Bigr\}.
$$
Applying the linear map $\pi$ and using linearity gives
$$
\pi \left(\sum_{i = 1}^N \alpha_i a_i\right)
 = 
\sum_{i = 1}^N \alpha_i \pi(a_i),
$$
with the same $\ell_1$ constraint on $(\alpha_i)$. Hence the image is exactly
$$
\Bigl\{\sum_{i = 1}^N \alpha_i \pi(a_i): \sum_{i = 1}^N|\alpha_i|\le 1\Bigr\}
 = 
\conv\{\pm \pi(a_1),\dots,\pm \pi(a_N)\}.
$$
\end{proof}

%%%%%%%%%%%%%%%%%%%%%%%%%%%%%%%%%%%%%%%%%%%
%%%%%%%%%%%%%%%%%%%%%%%%%%%%%%%%%%%%%%%%%%%
\section{Uniform suppression for the small-$U$ regime} \label{app:suppr-proof}

This appendix contains the probabilistic uniformization step in the small-$U$ regime, i.e. the proof of Lemma~\ref{lem:EK-prob}. The argument is a counting-and-union-bound scheme, but it is simpler than the ``all permutations'' uniform tilt step of \cite{Tikhomirov2019}: it uses only the block-tail estimate from Lemma~\ref{lem:block-tail-distances} and unions over subsets $J$ of a fixed cardinality.

%%%%%%%%%%%%%%%%%%%%%%%%%%%%%%%%%%%%%%%%%%%
\subsection{Enumerating possible $K$--coefficient submatrices}

Fix $p\in\{n-\tilde s,\dots,n\}$.
Let $\mathcal T_p$ be the collection of all $m\times p$ matrices $B$ which can appear as a submatrix
consisting of $p$ columns chosen from $F_1(A) = A^{(K)}$ for some $A\in\mathcal A_\varepsilon$.

Each such column $b\in\R^m$ satisfies
$$
\|b\|_1\le 1,\quad b\in(\varepsilon\Z)^m,\quad |\supp(b)|\le s.
$$
Indeed, $b$ is a $K$--part of a column of $A$, hence $b\in(\varepsilon\Z)^m$ and $\|b\|_1\le 1$; moreover,
if $b_j\neq 0$ then $|b_j|\ge 1/s$ by definition of the $K$--part, so $|\supp(b)|\le s$.
In particular, $\|b\|_2\le \|b\|_1\le 1$.

We claim that the number of possible such columns is at most $\big(\frac{C m}{s\varepsilon}\big)^s$.
To see this, first choose a support $J\subset[m]$ of size $r\le s$; the number of choices is
$$
\sum_{r = 0}^s \binom{m}{r}\le \Big(\frac{C m}{s}\Big)^s,
$$
and for a fixed support, each supported coordinate lies in $\varepsilon\Z\cap[-1,1]$, hence has at most $C/\varepsilon$ possibilities,
so there are at most $(C/\varepsilon)^s$ choices for the values. Multiplying yields the claim.

Therefore, since $B$ has $p$ columns,
\begin{equation} \label{eq:Tp-count-smallU-new}
|\mathcal T_p|\le \Big(\frac{C m}{s\varepsilon}\Big)^{sp}.
\end{equation}

%%%%%%%%%%%%%%%%%%%%%%%%%%%%%%%%%%%%%%%%%%%
\subsection{Proof of Lemma~\ref{lem:EK-prob}}

\begin{proof}[Proof of Lemma~\ref{lem:EK-prob}]
Fix the absolute constant $C_\star$ in Definition~\ref{def:EA1}. Choose $\tau_0\ge 1$ such that $\tau_0\sqrt{5/4}\le C_\star$ and $\tau_0$ is large enough for Lemma~\ref{lem:block-tail-distances} (with its constant $C$). Such a choice is possible by taking $C_\star$ sufficiently large once and for all.

Fix $p\in\{n-\tilde s,\dots,n\}$ and $B\in\mathcal T_p$ with $\rank(B) = p$.
For each $J\subset[p]$ with $|J| = r$, apply Lemma~\ref{lem:block-tail-distances} with $\tau = \tau_0$.
With $N = n-p + r$, we have $N\ge r$ and also $N\le (n-(n-\tilde s)) + r = \tilde s + r\le \frac54 r$ by \eqref{eq:EK-conds}.
Hence the conclusion of Lemma~\ref{lem:block-tail-distances} implies:
on the event $\mathcal E_{B,J}$,
at least $r/2$ of the last $r$ Gram--Schmidt distances are $\le \tau_0\sqrt N\le \tau_0\sqrt{5r/4}\le C_\star\sqrt r$.
Thus $\bigcap_{J:|J| = r}\mathcal E_{B,J}$ implies precisely the desired $C_\star\sqrt r$ bound for this $B$ and all $J$.

By \eqref{eq:block-tail-prob-allJ} and $\binom{p}{r}\le \binom{n}{r}$,
$$
\PP\Bigl(\bigcap_{J\subset[p], |J| = r}\mathcal E_{B,J}\Bigr)
\ge
1-\binom{n}{r}\exp\bigl(-c\tau_0^2Nr\bigr)
\ge
1-\binom{n}{r}\exp(-c'\tau_0^2 r^2).
$$
Using $\binom{n}{r}\le (en/r)^r\le n^r$ and absorbing constants,
$$
\PP\Bigl(\bigcap_{J\subset[p], |J| = r}\mathcal E_{B,J}\Bigr)
\ge 1-\exp\bigl(r\log n-c'' r^2\bigr).
$$

Union bound over all $B\in\mathcal T_p$ and $p\in\{n-\tilde s,\dots,n\}$.
Using \eqref{eq:Tp-count-smallU-new}, $m = n^3$, and $\varepsilon^{-1} = \rho n^2$, we have for $s\ge 1$,
$$
\log\Big(\frac{C m}{s\varepsilon}\Big)
 = 
\log\Big(\frac{C n^3}{s}\cdot \rho n^2\Big)
\le
C_0\log(n\rho)
$$
after adjusting $C_0$.
Hence, for $p\le n$,
$$
\log|\mathcal T_p|
\le
sp\log\Big(\frac{C m}{s\varepsilon}\Big)
\le
C_0sn\log(n\rho).
$$
Therefore the total failure probability is bounded by
\begin{align*}
\PP\bigl((\mathcal E^{(K)}(r,\tilde s))^c\bigr)
&\le
\sum_{p = n-\tilde s}^n |\mathcal T_p|\exp\bigl(r\log n-c''r^2\bigr)\\
&\le
(\tilde s + 1)\exp\Bigl(C_0sn\log(n\rho) + r\log n-c''r^2\Bigr).
\end{align*}

We use the parameter constraints \eqref{eq:EK-conds}. The condition $\tilde s^2/(ns)\ge C\log(n\rho)$ implies
$$
sn\log(n\rho)\le \frac{1}{C}\tilde s^2.
$$
Since $r\ge 4\tilde s$, we have $\tilde s^2\le r^2/16$.
Also $\tilde s\ge \log^2 n$ implies $r\ge 4\log^2 n$ and hence $r\log n = o(r^2)$ as $n\to\infty$.
Choosing the absolute constant $C$ in \eqref{eq:EK-conds} large enough and $n$ large enough, we obtain
$$
C_0sn\log(n\rho) + r\log n-c''r^2\le -c_1 r^2,
$$
for some $c_1>0$.
Thus
$$
\PP\bigl((\mathcal E^{(K)}(r,\tilde s))^c\bigr)\le (\tilde s + 1)e^{-c_1 r^2}\le \frac1n
$$
for all sufficiently large $n$, proving Lemma~\ref{lem:EK-prob}.
\end{proof}

%%%%%%%%%%%%%%%%%%%%%%%%%%%%%%%%%%%%%%%%%%%
%%%%%%%%%%%%%%%%%%%%%%%%%%%%%%%%%%%%%%%%%%%
\section{Verification of the parameter constraints} \label{sec:verify-constraints}

\begin{proof}[Proof of Lemma~\ref{lem:verify-constraints}]
Fix $n\ge n_0$ from~\eqref{eq:log-comparability}. Then $\rho\ge 1$ and
\begin{equation} \label{eq:Lambda-comp}
\log n \le \Lambda \le 2\log n.
\end{equation}

\subsubsection*{Preliminary bounds from ceilings/floors.}

From \eqref{eq:params-choice-47} and $n$ large enough, we have
\begin{equation} \label{eq:tilde-s-bounds-47}
n^{6/7}\Lambda^{2/7}\le \tilde s \le 2n^{6/7}\Lambda^{2/7}.
\end{equation}
Also $r\ge C_2\tilde s\Lambda$ and $r\le C_2\tilde s\Lambda + 1\le 2C_2\tilde s\Lambda$ for all large $n$
(since $\tilde s\Lambda\to\infty$).

Let
$$
A := \frac{\tilde s^2}{C_3 n\Lambda}.
$$
Using \eqref{eq:tilde-s-bounds-47}, we have
$$
A \ge \frac{(n^{6/7}\Lambda^{2/7})^2}{C_3 n\Lambda}
 = \frac{n^{12/7}\Lambda^{4/7}}{C_3 n\Lambda}
 = \frac{1}{C_3} n^{5/7}\Lambda^{-3/7}\xrightarrow[n\to\infty]{}\infty.
$$
Hence for $n\ge n_2$ large enough, $A\ge 2$ and therefore
\begin{equation} \label{eq:s-bounds-47}
\frac{A}{2}\le s = \lfloor A\rfloor \le A,
\quad\text{ i.e.}\quad
\frac{\tilde s^2}{2C_3 n\Lambda}\le s \le \frac{\tilde s^2}{C_3 n\Lambda}.
\end{equation}
In particular, $s\le n$ for all sufficiently large $n$. Indeed, using the upper bound $\tilde s\le 2n^{6/7}\Lambda^{2/7}$ from \eqref{eq:tilde-s-bounds-47},
$$
s \le \frac{\tilde s^2}{C_3 n\Lambda}
\le \frac{4n^{12/7}\Lambda^{4/7}}{C_3 n\Lambda}
= \frac{4}{C_3}n^{5/7}\Lambda^{-3/7}\le n
$$
for all sufficiently large $n$ (equivalently, $n^{2/7}\Lambda^{3/7}\to\infty$). Thus $s\in[1,n]$ as required in \S\ref{sec:bridge}.

\subsubsection*{1) Check \eqref{eq:EK-conds}.}

First, $r\ge C_2\tilde s\Lambda\ge 4\tilde s$ because $C_2\ge 4$ and $\Lambda\ge 1$.
Next, using $r\le 2C_2\tilde s\Lambda$ and \eqref{eq:tilde-s-bounds-47},
$$
\frac{r}{n}\le \frac{2C_2\tilde s\Lambda}{n}
\le \frac{4C_2 n^{6/7}\Lambda^{2/7}\Lambda}{n}
 = 4C_2 n^{-1/7}\Lambda^{9/7}\xrightarrow[n\to\infty]{}0,
$$
so for $n\ge n_3$ large enough we have $r\le n/4$. Since $\tilde s\le n/2$ for $n$ large (because $\tilde s/n\to 0$),
we get $r\le (n-\tilde s)/2$ for $n\ge n_3$.

Moreover, $\tilde s\ge n^{6/7}\ge (\log n)^2$ for all $n\ge n_4$.
Finally, from the upper bound on $s$ in \eqref{eq:s-bounds-47},
$$
\frac{\tilde s^2}{ns}\ge C_3\Lambda \ge C_{\mathrm{EK}}\Lambda = C_{\mathrm{EK}}\log(n\rho),
$$
since $C_3\ge C_{\mathrm{EK}}$. Thus \eqref{eq:EK-conds} holds for all $n$ large enough.

\subsubsection*{2) Check \eqref{eq:r-beats-comb}.}

We require
$$
c_{\mathrm{decay}} r \ge 4\tilde s\log\Big(\frac{en}{\tilde s}\Big) + \log 16.
$$
Since $r\ge C_2\tilde s\Lambda$ and $\Lambda\ge \log n$, we have
$$
c_{\mathrm{decay}}r \ge c_{\mathrm{decay}}C_2 \tilde s\log n.
$$
Also, $\tilde s\ge n^{6/7}$ implies $n/\tilde s\le n^{1/7}$ and hence
$$
\log\Big(\frac{en}{\tilde s}\Big)\le \log(e n^{1/7}) \le 2\log n
$$
for all large $n$. Therefore,
$$
4\tilde s\log\Big(\frac{en}{\tilde s}\Big)\le 8\tilde s\log n,
$$
and for $n$ large the additive term $\log 16$ is dominated by (say) $\tilde s$.
Thus it suffices that $c_{\mathrm{decay}}C_2\ge 16$, which is ensured by $C_2\ge 16/c_{\mathrm{decay}}$.

\subsubsection*{3) Check \eqref{eq:scale}.}

We need $4\rho C_\star\sqrt r\le c_{\mathrm{suppr}}n$.
Using $r\le 2C_2\tilde s\Lambda$ and \eqref{eq:tilde-s-bounds-47},
$$
\sqrt r \le \sqrt{2C_2} \sqrt{\tilde s\Lambda}
\le \sqrt{2C_2} \sqrt{2n^{6/7}\Lambda^{2/7}\cdot \Lambda}
 = 2\sqrt{C_2} n^{3/7}\Lambda^{9/14}.
$$
Hence
$$
4\rho C_\star\sqrt r
\le 8C_\star\sqrt{C_2} \rho n^{3/7}\Lambda^{9/14}.
$$
With $\rho = c_0n^{4/7}(\log n)^{-C}$ and $\Lambda\le 2\log n$,
$$
4\rho C_\star\sqrt r
\le
8C_\star\sqrt{C_2} c_0 n \Lambda^{9/14}(\log n)^{-C}
\le
C' c_0 n (\log n)^{9/14-C}.
$$
Since $C\ge 2$, one has $(\log n)^{9/14-C}\le 1$ for all $n\ge 3$.
Thus it suffices to choose $c_0^\ast>0$ so that $C'c_0^\ast\le c_{\mathrm{suppr}}$.
Then for $c_0\le c_0^\ast$, \eqref{eq:scale} holds for all $n\ge 3$.

\subsubsection*{4) Check \eqref{eq:rho-largeU-cond}.}

Recall $L = C_0\sqrt{\frac{n}{s}\Lambda}$ by \eqref{eq:L-def}.
From the lower bound on $s$ in \eqref{eq:s-bounds-47},
$$
\frac{n}{s} \le \frac{2C_3 n^2\Lambda}{\tilde s^2},
$$
hence
\begin{equation} \label{eq:L-upper-47}
L \le C_0\sqrt{\frac{2C_3 n^2\Lambda^2}{\tilde s^2}}
 = C_0\sqrt{2C_3} \frac{n\Lambda}{\tilde s}.
\end{equation}
Therefore the left-hand side of \eqref{eq:rho-largeU-cond} satisfies
$$
\frac{C_{\mathrm{U}}\rho L n\sqrt{\Lambda}}{\tilde s^2}
\le
C_{\mathrm{U}}\rho\cdot \Bigl(C_0\sqrt{2C_3}\frac{n\Lambda}{\tilde s}\Bigr)\cdot \frac{n\sqrt{\Lambda}}{\tilde s^2}
 = 
C'' \rho \frac{n^2\Lambda^{3/2}}{\tilde s^3},
$$
where $C'' := C_{\mathrm{U}}C_0\sqrt{2C_3}$.
Using $\tilde s\ge n^{6/7}\Lambda^{2/7}$ gives $\tilde s^3\ge n^{18/7}\Lambda^{6/7}$ and hence
$$
\frac{C_{\mathrm{U}}\rho L n\sqrt{\Lambda}}{\tilde s^2}
\le
C'' \rho \frac{n^{14/7}\Lambda^{3/2}}{n^{18/7}\Lambda^{6/7}}
 = 
C'' \rho n^{-4/7}\Lambda^{9/14}.
$$
Substituting $\rho = c_0n^{4/7}(\log n)^{-C}$ and using $\Lambda\le 2\log n$ yields
$$
\frac{C_{\mathrm{U}}\rho L n\sqrt{\Lambda}}{\tilde s^2}
\le
C'' c_0 \Lambda^{9/14}(\log n)^{-C}
\le
C''' c_0 (\log n)^{9/14-C}.
$$
Since $C\ge 2$, we have $(\log n)^{9/14-C}\le 1$ for all $n\ge 3$.
Thus it suffices to additionally choose $c_0^\ast$ small enough so that
$$
C''' c_0^\ast \le \frac{1}{4e^{2C_1 + 1}}.
$$
Then \eqref{eq:rho-largeU-cond} holds for all $n\ge 3$.

\subsubsection*{5) Check $\tilde s\ge \Lambda$.}

By \eqref{eq:tilde-s-bounds-47}, $\tilde s\ge n^{6/7}\Lambda^{2/7}$.
Since $\Lambda\le 2\log n$ and $n^{6/7}/\log n\to\infty$, there exists $n_5$ such that
$n^{6/7}\ge 2\log n$ for all $n\ge n_5$. Then for $n\ge n_5$,
$$
\tilde s \ge n^{6/7}\Lambda^{2/7}\ge n^{6/7}\ge 2\log n\ge \Lambda,
$$
which is \eqref{eq:tilde-s-ge-Lambda}.

Taking $n_1 := \max\{n_0,n_2,n_3,n_4,n_5,3\}$ completes the proof.
\end{proof}

%%%%%%%%%%%%%%%%%%%%%%%%%%%%%%%%%%%%%%%%%%%


\begin{thebibliography}{99}

\bibitem{BourgainSzarek1988}
J.~Bourgain and S.~J.~Szarek,
\newblock The Banach--Mazur distance to the cube and the Dvoretzky--Rogers factorization,
\newblock \emph{Israel J. Math.} \textbf{62} (1988), no.~2, 169--180.

\bibitem{BourgainTzafriri1987}
J.~Bourgain and L.~Tzafriri,
\newblock Invertibility of ``large'' submatrices with applications to the geometry of Banach spaces and harmonic analysis,
\newblock \emph{Israel J. Math.} \textbf{57} (1987), 137--224.

\bibitem{DavidsonSzarek2001}
K.~R.~Davidson and S.~J.~Szarek,
\newblock Local operator theory, random matrices and Banach spaces,
\newblock in \emph{Handbook of the Geometry of Banach Spaces}, Vol.~I,
North-Holland, Amsterdam, 2001, pp.~317--366.

\bibitem{DvoretzkyRogers1950}
A.~Dvoretzky and C.~A.~Rogers,
\newblock Absolute and unconditional convergence in normed linear spaces,
\newblock \emph{Proc. Natl. Acad. Sci. USA} \textbf{36} (1950), 192--197.

\bibitem{FriedlandYoussef2019}
O.~Friedland and P.~Youssef,
\newblock Approximating matrices and convex bodies,
\newblock \emph{Int. Math. Res. Not. IMRN} (2019), no.~8, 2519--2537.

\bibitem{Giannopoulos1995}
A.~Giannopoulos,
\newblock A note on the Banach--Mazur distance to the cube,
\newblock in \emph{Geometric Aspects of Functional Analysis (1992--1994)},
Oper. Theory Adv. Appl., vol.~77, Birkh\"auser, Basel, 1995, pp.~67--73.

\bibitem{Gluskin1981}
E.~D.~Gluskin,
\newblock Diameter of the Minkowski compactum is approximately $n$,
\newblock \emph{Funktsional. Anal. i Prilozhen.} \textbf{15} (1981), no.~1, 72--73;
English transl.: \emph{Funct. Anal. Appl.} \textbf{15} (1981), no.~1, 57--58.

\bibitem{NaorYoussef2017}
A.~Naor and P.~Youssef,
\newblock Restricted invertibility revisited,
\newblock in \emph{A Journey Through Discrete Mathematics: A Tribute to
Ji\v{r}\'i Matou\v{s}ek},
(M.~Loebl, J.~Ne\v{s}et\v{r}il, and R.~Thomas, eds.),
Springer, Cham, 2017, pp.~657--691.

\bibitem{SzarekTalagrand1989}
S.~J.~Szarek and M.~Talagrand,
\newblock An isomorphic version of the Sauer--Shelah lemma and the Banach--Mazur distance to the cube,
\newblock in \emph{Geom. Aspects of Funct. Anal. (1987--88)},
Lecture Notes in Math., vol.~1376, Springer, Berlin, 1989, pp.~105--112.

\bibitem{Szarek1990}
S.~J.~Szarek,
\newblock Spaces with large distance to $\ell_\infty^n$ and random matrices,
\newblock \emph{Amer. J. Math.} \textbf{112} (1990), no.~6, 899--942.

\bibitem{Tikhomirov2019}
K.~Tikhomirov,
\newblock On the Banach--Mazur distance to cross-polytope,
\newblock \emph{Adv. Math.} \textbf{345} (2019), 598--617.

\bibitem{Youssef2014}
P.~Youssef,
\newblock Restricted invertibility and the Banach--Mazur distance to the cube,
\newblock \emph{Mathematika} \textbf{60} (2014), no.~2, 237--260.

\bibitem{PisierMaurey1981}
G.~Pisier,
\newblock Remarques sur un r\'esultat non publi\'e de B.~Maurey,
\newblock \emph{S\'eminaire d'Analyse Fonctionnelle} (Maurey--Schwartz),
1980--1981, Expos\'e no.~5, 1--12.

\bibitem{LedouxTalagrand1991}
M.~Ledoux and M.~Talagrand,
\newblock \emph{Probability in Banach Spaces: Isoperimetry and Processes},
\newblock Ergebnisse der Mathematik und ihrer Grenzgebiete (3), vol.~23,
Springer-Verlag, Berlin, 1991.

\bibitem{Ledoux2001}
M.~Ledoux,
\newblock \emph{The Concentration of Measure Phenomenon},
\newblock Mathematical Surveys and Monographs, vol.~89, AMS, 2001.

\end{thebibliography}
\end{document}